\documentclass[1p]{elsarticle}
\usepackage{mystyle}
\usepackage{amsmath}
\usepackage[colorlinks=true, citecolor=cyan, linkcolor=blue, urlcolor=cyan, filecolor=magenta]{hyperref}
\numberwithin{equation}{section}

\makeatletter
\def\ps@pprintTitle{
 \let\@oddfoot\@empty
 \let\@evenfoot\@empty
}
\makeatother
\begin{document}
\title{Type-\RomanNumeralCaps{1} Blowup Solutions for Yang--Mills Flow}
\author[snu]{Jaehwan Kim}
\ead{jhhope1@snu.ac.kr}
\author[kias]{Sanghoon Lee}
\ead{sl29@kias.re.kr}
\address[snu]{\scriptsize Seoul National University, 1 Gwanak-ro, Gwanak-gu, Seoul 08826, Republic of Korea.}
\address[kias]{\scriptsize Korea Institute for Advanced Study, 85 Hoegiro, Dongdaemun-gu, Seoul 02455, Republic of Korea.}

\begin{abstract}
    In this paper, we construct an infinite-dimensional family of solutions for the Yang--Mills flow on $\RR^n \times SO(n)$ for $5 \leq n \leq 9$, which converge to $SO(n)$-equivariant homothetically shrinking solitons, modulo the gauge group.
    As a corollary, we prove the existence of asymmetric Type-\RomanNumeralCaps{1} blowup solutions for the Yang--Mills flow.
\end{abstract}
\maketitle

\section{Introduction}
\subsection{Background for the Yang--Mills flow}
Let $E$ be a vector bundle of rank $r$ with a metric over a Riemannian manifold $(M, g)$ of dimension $n$.
Denote by $A$ the connection 1-form of a metric connection on $E$.
In a local chart $(x_1, \ldots, x_n) \in U \subseteq \RR^n$ with an orthonormal frame $\Set{e_1, \ldots, e_r}$, the connection 1-form $A$ is expressed as $A = A_i dx^i$ where $A_i: U \to \mf{so}(r)$.
Let $F_A$ denote the curvature form of $A$.
The curvature form $F_A$ is given explicitely by $F_A = \frac{1}{2} F_{ij} dx^i \wedge dx^j$ where
\begin{align*}
    F_{ij} = \partial_i A_j - \partial_j A_i + [A_i, A_j]: U \to \mf{so}(r).
\end{align*}
The Yang--Mills energy functional is defined as
\begin{align*}
    \text{YM}(A) = \frac{1}{2}\int_M \|F_A\|^2 d\text{Vol}_g.
\end{align*}

Let $D_A$ denote the covariant differential associated with $A$ and $D_A^*$ denote its adjoint.
Introducing a time dependence of $A$, the Yang--Mills flow is defined as the gradient flow of the Yang--Mills energy functional:
\begin{align} \label{eq:YM-flow}
    \tfrac{\partial}{\partial t} A(x, t) = -D_A^* F_A(x, t).
\end{align}

The Yang--Mills flow has been extensively studied in various geometric settings. For a more comprehensive exposition, we refer readers to Feehan's monograph \cite{Feehan}. Here, we focus on the dynamical properties of the Yang--Mills flow.

R{\aa}de \cite{Rade} established long-time existence results for the Yang--Mills flow on two- and three-dimensional manifolds. In the four-dimensional case, which is the critical dimension, Struwe \cite{Struwe} and Schlatter \cite{Schlatter-long-time-behavior-4dim} investigated global weak solutions of the Yang--Mills flow, allowing for the possibility of finitely many point singularities. Long-time existence was established by Schlatter \cite{Schlatter-long-time-existence-small-data} for the case of small initial data, and by Hong and Tian \cite{Hong-Tian} as well as Schlatter, Struwe, and Tahvildar-Zadeh \cite{Schlatter} for equivariant cases. More recently, Waldron \cite{Waldron} resolved the general long-time existence problem (see also \cite{Waldron-1, Waldron-2}). Additionally, the behavior of the flow in infinite time is explored in \cite{Sire-Wei-Zheng}.

However, in supercritical dimensions $n \geq 5$, finite-time blowup may occur.
The asymptotic behavior and the singular set structure of the Yang-Mills flow in dimensions $n \geq 4$ were studied in \cite{Hong-Tian-Asymptotical-behavior}.
Finite time blowup for $SO(n)$-bundle over $S^n$ was studied by Naito \cite{Naito}, and for $SO(n)$-bundle over $\RR^n$ by Grotowski \cite{Grotowski} and Gastel \cite{Gastel}.
Weinkove \cite{Weinkove} explored Type-\RomanNumeralCaps{1} singularities of the Yang--Mills flow, showing that the blowup subsequence converges to a \emph{homothetically shrinking soliton} up to gauge transformation. We would like to note further that the existence of self-similar blowup solutions in higher dimensions $n \ge 10$ was excluded by Bizo\'{n} and Wasserman \cite{Bizon-Wasserman}.

A homothetically shrinking soliton centered at $(0, 1) \in \RR^n \times \RR$, or simply a \emph{soliton}, is a solution $A_j: \RR^n \times (-\infty, 1) \to \mf{so}(n)$ of the Yang--Mills flow \eqref{eq:YM-flow} on a Euclidean bundle over $\RR^n$ that satisfies the following scale invariance property
\begin{align*}
    A_j (x,1 - s) = \lambda A_j(\lambda x,1 - \lambda^2s)
\end{align*}
for all $\lambda > 0$ and $s > 0$.
If we denote $A_0$ as the connection at $t = 0$ of a soliton $A$, its curvature form satisfies
\begin{align} \label{eq:soliton-curvature}
    D_{A_0}^*F_{A_0} + \tfrac{1}{2} x \cdot F_{A_0} = 0,
\end{align}
where $x \cdot $ denotes the interior product by the position vector $x$.
An example of a soliton was given in \cite{Weinkove} for the rank $n$ Euclidean bundle over $\RR^n$ for $5 \leq n \leq 9$.
Explicitly, the soliton $\bf{W} = \bf{W}_i dx^i$, where $\bf{W}_i: \RR^n \times (-\infty, 1) \to \mf{so}(n)$, is given by
\begin{align*}
    \mathbf{W}_i(x, t) = \frac{1}{\sqrt{1-t}} W_i(\frac{x}{\sqrt{1-t}}).
\end{align*}
Here, $W = W_i dx^i$ and $\sigma = \sigma_i dx^i$, where $W_i, \sigma_i: \RR^n \to \mf{so}(n)$ are defined as
\begin{align*}
    W_i(y) = \frac{\sigma_i(y)}{a|y|^2+b}, \quad (\sigma_{i})_\nu^\mu(y) = \delta_{i\nu} y^\mu - \delta_{i\mu} y^\nu,
\end{align*}
with constants
\begin{align*}
    a = \frac{\sqrt{n - 2}}{2\sqrt{2}}, \quad b = \frac{1}{2}(6n - 12 - (n+2)\sqrt{2n-4}).
\end{align*}

By analogy with the work of Colding and Minicozzi on the mean curvature flow(MCF) \cite{CM-genericMCF}, Chen and Zhang \cite{Chen-Zhang-entropy}, as well as Kelleher and Streets \cite{Kelleher-Street-entropy}, investigated the entropy stability of solitons.
Kelleher and Streets \cite{Kelleher-generel-blowup} utilized the entropy to provide a general description of blowup solutions for the Yang--Mills flow on closed Riemannian manifolds in dimensions $n\geq 4$.
They showed that, up to gauge transformations, the flow converges to either a soliton or a Yang--Mills connection as a blowup limit at singularities.

Under the $SO(n)$-equivariance condition, a connection $A(x, t)$ can be expressed as the multiplication of $\sigma(x)$ and a scalar function $w(|x|, t)$.
This formulation reduces the tensor-valued PDE system on $\RR^n$ to a scalar-valued PDE on $[0, \infty)$.
For the $SO(n)$-equivariant case, stability near the soliton $\bf{W}$ without any gauge transformation was demonstrated by Donninger and Sch\"orkhuber \cite{DONNINGER} for $n = 5$, and later extended by Glogi\'c and Sch\"orkhuber \cite{GLOGIC} for $5 \leq n \leq 9$ using the spectral analysis of the linearized operator of the flow.
However, without $SO(n)$-equivariance condition, no general stability results or even the construction of non-$SO(n)$-equivariant solutions converging to the soliton have been addressed.

\subsection{Spectral analysis on geometric flows and main theorem}
We briefly outline the method used in this paper to construct a solution $u$ of a forward parabolic PDE that converges exponentially to a stationary solution as $t \rightarrow \infty$. Let $L$ denote the time-independent linearized operator of the flow at the stationary solution $u_0$, and let $\mc N$ represent the nonlinear term.
Suppose that $-L$ admits a discrete spectral decomposition in an appropriate $L^2$-space.

Our goal is to find an exponentially decaying solution $v(x, t)$ that satisfies the equations
\begin{align*}
    \tfrac{\partial}{\partial t} v = L & v + \mc N (v), \\
    \Pi_{> 0} (v(\cdot, 0)             & - v_0) = 0
\end{align*}
for some initial condition $v_0$, where $\Pi_{> 0}$ is a spectral projector onto the positive eigenspace of $-L$. It is important to note that specifying $\tilde{v}(\cdot, 0) = v_0$ without applying the projection may lead to a solution $\tilde{v}$ that fails to decay to zero, as negative eigenfunctions could dominate the dynamics. By projecting, we ensure that the positive eigenfunctions control the limiting behavior of the solution as $t \rightarrow \infty$.

To achieve this, we first study the initial value problem for the corresponding linear PDE of the form $\big( \frac{\partial }{\partial t}- L \big) v =  h$. After obtaining a priori estimates for the linear problem, we introduce a parabolic H\"older space with appropriate weights to guarantee the desired decay of solutions in both space and time. Subsequently, we estimate the nonlinear error term $\mc N(v)$ in terms of $v$ and use a fixed-point argument to demonstrate the existence of a solution when the initial data $v_0$ has a sufficiently small norm. This approach yields a solution $u = u_0 + v$ that converges exponentially to $u_0$.

In this paper, we construct a family of converging solutions to the Yang--Mills flow near the soliton $\bf{W}$ without any symmetry condition.
We will precisely define the linearized operator \( L \) of the rescaled Yang--Mills de-Turck flow at the soliton in a Gaussian-weighted \( L^2 \)-space and the spectral projector $\Pi_{>0}$ in Section \ref{sec:preliminaries}.
\begin{theorem} \label{thm:main-theorem-intro}
    Fix $5 \leq n \leq 9$.
    For every smooth connection $\tilde{A}_0$ sufficiently close to \( W \), there exists a smooth solution $A$ to the Yang--Mills flow such that
    \begin{align*}
        \Pi_{>0}(A(\cdot, 0) - W - \tilde{A}_0) = 0,
    \end{align*}
    and converges to the soliton $\mathbf{W}$ up to gauge transformation.
\end{theorem}
See Theorem \ref{thm:main-theorem} for more precise statements.
As a corollary of the main theorem, we construct a non-$SO(n)$-equivariant converging solution:
\begin{corollary} \label{cor:non-equivariant-solution}
    There exists a non-$SO(n)$-equivariant solution to the Yang--Mills flow that converges to the soliton $\bf{W}$ up to gauge transformation.
\end{corollary}
Here, a solution is considered non-$SO(n)$-equivariant if it is not gauge equivalent to any $SO(n)$-equivariant solution.

The spectral analytic approach has been widely adapted to study the dynamics of various geometric flows.
For instance, Angenent and Vel{\'a}zquez \cite{Angenent-Degenerate-Neckpinches} constructed a rotationally symmetric solution exhibiting a Type-\RomanNumeralCaps{2} blowup by perturbing a cylinder slightly in the direction of an eigenfunction of a rescaled MCF in a compact region, and then capping it off.
While spectral analysis plays a key role in such constructions, it has also been applied to analyze asymptotic behaviors of solutions.
In a pioneering work, Sesum \cite{Sesum-Ricci} established the exponential convergence of the Ricci flow near a stable Ricci flat metric by analyzing the linearized operator associated with the Ricci-DeTurck flow on a closed manifold.
Furthermore, Sesum obtained the optimal convergence rate of the rescaled MCF of a smooth, convex surface toward a round sphere in \cite{Sesum-MCF}(see also \cite{Sesum-convergence-ricci-soliton, Sesum-limiting-behavior}).
It is worth noting that the round sphere is dynamically unstable, but entropy-stable \cite{CM-genericMCF}, and the flow can be stabilized by adjusting its rescaling center, which corresponds to adding negative eigenfunctions of the linearized operator.

These spectral analysis approaches have been applied in the classification of ancient solutions to various geometric flows.
An ancient solution of a flow is a solution defined for all $t \in (-\infty, T)$.
In the MCF, Angenent, Daskalopoulos, and Sesum \cite{ADS-asymptotics-MCF} established the precise asymptotics of symmetric ancient solutions by linearizing the flow around the round cylinder and showing that the eigenfunction associated with eigenvalue zero dominates in the asymptotically cylindrical region.
Brendle and Choi \cite{Brendle-Choi-1, Brendle-Choi-2} analyzed strictly convex, non-collapsed, noncompact ancient solutions of the MCF by linearizing the flow around the cylinder, showing exponential convergence as $t \to -\infty$.
Using the neck improvement theorem, they proved these solutions are unique and coincide with the bowl soliton.
The neck improvement theorem and related estimates \cite{ADS-asymptotics-MCF} also helped classify two-convex, closed ancient solutions \cite{ADS-TWO-CONVEX}.
For further results based on spectral analysis at cylinders, see \cite{CHHW-ancient-asymptotic-cylinder, CHH-Ancient-low-entropy, du2024hearing, choi2022classification, choi2021nonexistence}.

In the Ricci flow, the analysis in the cylindrical region has been applied to study $\kappa$-noncollapsed ancient solutions.
This includes $\kappa$-noncollapsed ancient solutions on spheres in dimensions three and higher \cite{Brendle-Daskalopoulos-Naff-Sesum, BDS-compact-ancient-Ricci-flow, ABDS-Ricci-cylindrical-asymptotic}, as well as noncompact $\kappa$-noncollapsed ancient solutions \cite{Brendle,Brendle-Naff}, extending results from lower-dimensional cases (one or two dimensions) \cite{DHS-classification-compact-ancient-curve-shortening-flow, DHS-classification-compact-ancient-Ricci-flow}.

Using negative eigenvectors of the Jacobi operator of a minimal surface, Choi and Mantoulidis \cite{Choi} constructed a family of ancient MCF exponentially converging to the given minimal surface.
In this case, the negative eigenfunctions dominate the dynamics of the flow as $t \rightarrow -\infty$.
A similar analysis comprises one of the key elements of \cite{CHODOSH,CHODOSH-MCF-2} for studying singularity formulation of generic MCF.
When a non-trivial zero eigenfunction governs the overall behavior of the flow, polynomial convergence may occur, as constructed by Carlotto, Chodosh, and Rubinstein \cite{Carlotto} in the case of Yamabe flow.
The asymptotic behaviors of ancient ovals in the MCF \cite{ADS-asymptotics-MCF} and three-dimensional Ricci flow \cite{ABDS-Ricci-cylindrical-asymptotic} also exhibit polynomial convergence over time, driven by the dominance of a neutral eigenfunction.
More recently, Choi and Hung \cite{ChoiB} resolved Thom's gradient conjecture through a detailed analysis of the dominance of the neutral mode in slowly converging solutions.

\subsection{Comments on the method of proof}

Our main challenges arise from the fact that our PDE is defined on a non-compact domain and involves tensor-valued rather than scalar-valued functions.
As discussed in Section \ref{subsec:weak-sense-estimates}, we initially work with solutions in the Gaussian \( L^2 \)-space, which exhibit exponential decay in time but lack guaranteed spatial decay.
Although local Schauder \( L^2 \)-estimates (Theorem \ref{thm:parabolic-Schauder-Lp-estimate}) provide \( C^{2, \alpha} \)-regularity within compact subsets, extending these estimates to the entire domain is highly nontrivial due to the unbounded coefficients of the linearized operator. Achieving global \( C^{2, \alpha} \)-estimate is crucial for obtaining the necessary nonlinear estimates.

To overcome this difficulty, we employ a refined version of Kato's inequality (Lemma \ref{lem:Kato-inequality}) to obtain a global $C^0$-estimate for tensor-valued functions.
In ancient MCF, which deals with scalar functions, estimates are derived by comparison with a barrier function as in \cite{ADS-asymptotics-MCF}(see also \cite[Lemma 3.15]{CHODOSH2}).
In the Yang--Mills setting, the analysis involves tensor-valued quantities, necessitating control over the tensor norm. Since direct comparison of tensor-valued functions is not feasible, we instead compare their norms, making the use of Kato's inequality indispensable.
In contrast to other versions of Kato's inequality, such as \cite[Lemma 3.1]{Uhlenbeck-Kato}, where the Cauchy-Schwarz inequality is applied early in comparing \( |\nabla u| \) and \( |\nabla |u|| \), our version of Kato's inequality carefully preserves Lie bracket terms, which account for the difference between \( |\nabla u| \) and \( |\nabla |u|| \) and enable more refined estimates.

Estimating global \( C^{2, \alpha} \)-norm in our setting differs from the case of ancient flows, as we are dealing with an initial boundary value problem. For instance, in ancient MCF, Lemma 3.3 of \cite{CHODOSH} bounds the spatially weighted H\"older norm of an ancient solution using a standard rescaling argument. However, in our case, the flow does not exist for \( t \to -\infty \), making the rescaling method inapplicable.

To address this issue, we introduce a weighted parabolic H\"older norm, applying the same weights to all derivative terms, unlike the stronger weights used in the existing norm. This adjusted norm allows us to bound the solution in the initial boundary value problem using a rescaling argument. Despite the weaker norm, nonlinear estimates can still be obtained through interpolation inequalities, although the H\"older exponent $\alpha$ must remain sufficiently small.
Estimates for the initial value problem involve terms containing norms of the initial data. However, to construct a solution with the desired decay properties, we need to perturb the component of the initial data in the non-positive eigenspace while keeping the positive eigenspace component fixed. This makes it crucial to estimate the decay behavior of eigenfunctions.

We successfully obtained these estimates by applying the refined Kato's inequality to the solution of the parabolic PDE of the form $e^{-\lambda t}\xi$, generated by a single eigenfunction of the linearized operator. This approach allowed us to derive sharp upper bounds on the growth of eigenfunctions, extending and refining the bounds previously established for the lowest eigenfunction in \cite[Proposition 4.1]{Bernstein}.

It is important to note that eigenfunctions corresponding to eigenvalues greater than $1/2$ may not be bounded. As a result, the main theorem considers initial conditions in a weighted H\"older space with appropriate growth conditions rather than restricting them to finite linear combinations of eigenfunctions, as is typically done in the case of ancient flows.

\subsection{Organization of the paper}
In Section \ref{sec:preliminaries}, we introduce the rescaled Yang--Mills de-Turck flow and the associated linearized operator. We provide a brief overview of the spectral properties of the linearized operator and present explicit eigenfunctions, including those corresponding to time and spatial translations.

In Section \ref{sec:estimates-for-linearized-de-turck-flow}, we analyze the linearized rescaled Yang--Mills de-Turck flow equation near the soliton. Specifically, in Sections \ref{subsec:weak-sense-estimates} and \ref{subsec:pointwise-estimates}, we derive various estimates for solutions of the linearized equation and establish nonlinear estimates. Estimates for the eigenfunctions of the linearized operator are provided in Section \ref{subsec:eigenfunction-estimates}.

We apply these estimates to construct solutions to the Yang--Mills flow, proving the main theorem and its corollary in Section \ref{sec:main-theorem}.

\section*{Acknowledgements}

The authors express their sincere gratitude to Professor Kyeongsu Choi for suggesting the problem and providing valuable insights and discussions. The second author was supported by KIAS Individual Grant, with grant number: MG096101.

\section{Preliminaries} \label{sec:preliminaries}

\subsection{Formulation of the rescaled Yang--Mills de-Turck flow}
We introduce the argument in \cite[Section 4]{Struwe} (see also \cite{de-Turck, Donaldson}) to make a linearized flow elliptic.
Let $\psi(x, t) = \mathbf{W}(x, t) + \phi(x, t)$ denote the solution of the Cauchy problem with the initial condition $A_0$ as
\begin{align} \label{eq:De-Turck-YM-flow}
    \begin{cases*}
        \tfrac{\partial}{\partial t} \psi + D_\psi^*F_\psi + D_\psi D_\psi^* \phi = 0, \\
        \psi|_{t=0} = A_0.
    \end{cases*}
\end{align}
Through the identification
\begin{align} \label{eq:gauge-transformation-ode}
    S^{-1}\tfrac{\partial}{\partial t} S = -D_\psi^* \phi,
\end{align}
the solution $\phi(x, t)$ induces a gauge transformation $S(x, t)$, which can be determined from the ODE with the initial condition $S(x, 0) = \Id$ for all $x \in \RR^n$.
If we define $A$ as a gauge transformation of $\psi$ by $S^{-1}$ as
\begin{align*}
    A \defeq (S^{-1})^*\psi = S \psi S^{-1} + SdS^{-1},
\end{align*}
then the connection $A$ satisfies the Yang--Mills flow equation
\begin{align*}
    \tfrac{\partial}{\partial t} A = -D_A^* F_A.
\end{align*}
The equation \eqref{eq:De-Turck-YM-flow} is called the \emph{Yang--Mills de-Turck flow}.

We introduce similarity coordinates $(y, \tau) \in \RR^n \times \RR$ defined as
\begin{align*}
    y = \frac{x}{\sqrt{1 - t}}, \quad \tau = -\log(1 - t),
\end{align*}
to leverage the scale invariance of the soliton.
By setting
\begin{align*}
    u_i(y, \tau) & \defeq e^{-\tau/2}\phi_i(e^{-\tau/2}y, 1 - e^{-\tau}),                         \\
    B_i(y, \tau) & \defeq e^{-\tau/2}\psi_i(e^{-\tau/2}y, 1 - e^{-\tau}) = W_i(y) + u_i(y, \tau),
\end{align*}
equation \eqref{eq:De-Turck-YM-flow} transforms into
\begin{align}
    \tfrac{\partial}{\partial \tau}u_j & = \tfrac{\partial}{\partial \tau} B_j  \notag                                                                                                           \\
                                       & =  \Delta B_j - \tfrac{1}{2}(y \cdot \nabla B)_j - \tfrac{1}{2}B_j + \partial_i [B_i, B_j] + [B_i, \partial_i B_j - \partial_j B_i + [B_i, B_j]] \notag \\
                                       & \quadn{10}+ \partial_j[B_i, u_i] + [B_j, \partial_i u_i + [B_i, u_i]] \notag                                                                            \\
                                       & = L(u)_j + \mc N(u)_j \label{eq:transformed-YM-flow-3}
\end{align}
where $L$ is a linearized operator and $\mc N$ is a nonlinear operator defined by
\begin{align}
    L(u)_j     & \defeq \Delta u_j -\tfrac{1}{2}(y \cdot \nabla u)_j - \tfrac{1}{2}u_j + 2[W_i, \partial_i u_j] + [W_i, [W_i, u_j]]\notag \\
               & \quadn{15} +2[u_i, \partial_i W_j - \partial_j W_i + [W_i, W_j]] \label{eq:de-turck-linearized-operator}                 \\
               & = -(D_W D_W^* + D_W^* D_W + \tfrac{1}{2}) u_j - y^i {\nabla_W}_i u_{j} - [u_i, {F_W}_{ij}], \notag                       \\
    \mc N(u)_j & \defeq 2[u_i, \partial_i u_j + [W_i, u_j]] - [u_i, \partial_ju_i + [W_j, u_i]] - [u_i, [u_j, u_i]].\notag
\end{align}

\subsection{Spectral properties of the linearized operator} \label{subsec:spectral-properties}
Denote $\Omega^1(\mf{so}(n))$ the space of $\mf{so}(n)$-valued 1-forms on $\RR^n$.
We check a spectral property of the linear operator $L$ in $L^2_{\boldsymbol{\rho}}(\Omega^1(\mf{so}(n)))$ where $\boldsymbol{\rho}(x) = e^{-|x|^2/4}$ is a weight function.
Let $U: L^2(\Omega^1(\mf{so}(n))) \to L^2_{\boldsymbol{\rho}}(\Omega^1(\mf{so}(n))); u(\cdot) \mapsto e^{|\cdot|^2/8}u(\cdot)$ be a unitary operator.
Then, the operator $-L$ transforms into an operator $\mc A$ on $L^2(\Omega^1(\mf{so}(n)))$ as
\begin{align}
    (\mc A \varphi)_j & \defeq -(U^{-1} L U \varphi)_j                                  \notag                                                                             \\
                      & = -\big(\Delta \varphi_j+ (\tfrac{1}{16}(4(n-2) - |y|^2) - \tfrac{1}{2})\varphi_j + 2[W_i, \partial_i \varphi_j] + [W_i, [W_i, \varphi_j]]  \notag \\
                      & \quadn{15}+ 2[\varphi_i, \partial_i W_j - \partial_j W_i + [W_i, W_j]]\big)             \notag                                                     \\
                      & = (D_WD_W^* + D_W^*D_W)\phi_j + \tfrac{1}{16}(|y|^2 - 4(n-2))\phi_j - [\phi_i, {F_W}_{ij}] \label{eq:schrodinger-operator-1}.
\end{align}

The operator $L$ is self-adjoint on $L^2_{\boldsymbol{\rho}}$ as each term of $\mc A$ is self-adjoint on $L^2$ in the expression \eqref{eq:schrodinger-operator-1}.
Moreover, we have:
\begin{lemma} \label{lem:spectral-properties}
    The linear operator $-L$ has a discrete spectrum $\lambda_1 \leq \lambda_2 \leq \cdots$ with $\lim_i \lambda_i = \infty$ and a corresponding complete $L_{\boldsymbol{\rho}}^2$-orthonormal set $\xi_1, \xi_2, \cdots \in L^2_{\boldsymbol{\rho}}(\Omega^1(\mf{so}(n)))$ such that $L \xi_i = -\lambda_i \xi_i$.
    \begin{proof}
        It is enough to prove the same statement for the operator $-\mc A$ in $L^2$.
        Since $\|{(F_W)_{ij}}^{\mu}_{\nu}\|_{L^\infty}$ and $\|{W_i}^{\mu}_{\nu}\|_{L^\infty}$ are bounded, the operator $\mc A$ is bounded below from the expression \eqref{eq:schrodinger-operator-1}.
        We use $\|\cdot\|$ to denote the $L^2$ norm in the proof for brevity.
        By \cite[Theorem \RomanNumeralCaps{13}.64]{METHODS-OF-MODERN-MATHEMATICAL-PHYSICS}, it suffices to show that
        \begin{align*}
            F_{\mc A} \defeq \Set{u \in \mc Q(A)\given \|u\| \leq 1, \langle u, \mc A u \rangle \leq b},
        \end{align*}
        where $\mc Q(A)$ denotes the form domain of $\mc A$ and $b > 0$, is compact.
        Let $V(y) = \frac{1}{16}(|y|^2 - 4(n - 2))$.
        From \eqref{eq:schrodinger-operator-1} and H\"older's inequality, we obtain
        \begin{align*}
            b & \geq \langle u, \mc A u \rangle                                                                  \\
              & = \|D_W u \|^2 + \| D_W^* u \|^2 + \langle u, V u \rangle - \langle u_j, [u_i, F_{ij}] \rangle   \\
              & \geq \tfrac{1}{2}(\|d u\|^2 + \|d^* u \|^2) + \tfrac{1}{16}\langle u, |y|^2 u \rangle - C\|u\|^2 \\
              & \geq \tfrac{1}{2}\sum_{i,j}\|\partial_i u_j\|^2 +  \tfrac{1}{16}\langle u, |y|^2 u \rangle - C,
        \end{align*}
        for $u \in F_{\mc A}$ and some constant $C$.
        The inequalities
        \begin{align*}
            \langle \hat{{u_j}^\mu_\nu}, p^2\hat{{u_j}^\mu_\nu}\rangle = \sum_i \|\partial_i {u_j}^\mu_\nu\|^2 \leq 2(b + C) \quad \text{ and } \quad \langle {u_j}^\mu_\nu, \tfrac{1}{16}|y|^2 {u_j}^\mu_\nu\rangle \leq b + C.
        \end{align*}
        imply that ${u_j}^\mu_\nu$ is contained in a compact subset of $L^2(\RR^n)$ by the Rellich's criterion(see \cite[Theorem \RomanNumeralCaps{13}.65]{METHODS-OF-MODERN-MATHEMATICAL-PHYSICS}).
    \end{proof}
\end{lemma}
Since $\lambda_i \to \infty$ as $i \to \infty$, there exists an index $I$ such that
\begin{align*}
    \lambda_1 \leq \cdots \leq \lambda_I \leq 0 < \lambda_{I+1} \leq \cdots.
\end{align*}
We define the spectral projector $\Pi_{\sim \mu}: L_{\boldsymbol{\rho}}^2(\Omega^1(\mf{so}(n))) \to L_{\boldsymbol{\rho}}^2(\Omega^1(\mf{so}(n)))$ for a binary relation $\sim \in \Set{=, \neq, <, >, \leq, \geq}$ as:
\begin{align*}
    \Pi_{\sim \mu} : f \mapsto \sum_{j: \lambda_j \sim \mu} \langle f,  \xi_j \rangle_{\rho} \xi_j.
\end{align*}
In the remainder of the section, we present some eigenpairs of $-L$ explicitly that correspond to the time or spatial translations of the soliton $W$.
First, an eigenfunction $\mathbf{g}$ corresponds to the time translation with eigenvalue $-1$ is given by
\begin{align*}
    \mathbf g_j(y) \defeq C(y^i\partial_i W_j + W_j) = C'\frac{\sigma_j(y)}{(a\rho^2 + b)^2}.
\end{align*}

Note that the eigenfunction $\mathbf{g}$ of the de-Turck flow corresponds to the lowest eigenfunction of the linearized operator of equivariant Yang--Mills flow, as stated in \cite[(3.2)]{DONNINGER}.
This coincidence occurs because the $D_W^* \mathbf{g}$ term vanishes in the equivariant setting, resulting in the de-Turck flow equation being the same as the equivariant Yang--Mills flow equation.
However, this is not true in the general case.

We now present eigenfunctions corresponding to the spatial translations of the soliton $W$.

\begin{lemma}
    $-L$ has eigenfunctions $\Set{{F_W}_\alpha}_{\alpha \in [1, n]}$ with eigenvalue $-1/2$ where
    \begin{align*}
        ({F_W}_{\alpha})_j = \partial_\alpha W_j - \partial_j W_\alpha + [W_\alpha, W_j].
    \end{align*}
    \begin{proof}
        For brevity, we will write $D$ for $D_W$ and $F$ for $F_W$ in this proof.
        The terms $D^*D F_\alpha$ and $DD^*F_\alpha$ are calculated as
        \begin{align*}
            D^*D (F_\alpha)_j & = D^*(D_i F_{\alpha j} + D_j F_{i \alpha})                                                                      \\
                              & = D^*(D_\alpha F_{ij})\qquad (\text{2nd Bianchi identity})                                                      \\
                              & = -\partial_i(\partial_\alpha F_{ij} + [W_\alpha, F_{ij}]) - [W_i, \partial_\alpha F_{ij} + [W_\alpha, F_{ij}]) \\
                              & = -D_\alpha (\frac{1}{2}x^i F_{ij}) - [F_{i\alpha}, F_{ij}]                                                     \\
                              & = -\frac{1}{2}x^iD_\alpha F_{ij}-\frac{1}{2}F_{\alpha j} - [F_{i\alpha}, F_{ij}],                               \\
            DD^*(F_\alpha)_j  & = -D_j(\frac{1}{2}x^iF_{\alpha i})                                                                              \\
                              & = -\frac{1}{2}x^i D_j F_{\alpha i} - \frac{1}{2}F_{\alpha j}
        \end{align*}
        where we used the relation \eqref{eq:soliton-curvature}.
        Then, we obtain
        \begin{align*}
            L (F_\alpha)_j & = -(D^*D + DD^*)(F_\alpha)_j - \frac{1}{2}(F_\alpha)_j - \frac{1}{2}x^i D_i F_{\alpha j} + [(F_\alpha)_i, F_{ij}] \\
                           & = \frac{1}{2} F_{\alpha j} \qquad (\text{2nd Bianchi identity}). \qedhere
        \end{align*}
    \end{proof}
\end{lemma}

\section{Estimates for rescaled Yang--Mills de-Turck flow} \label{sec:estimates-for-linearized-de-turck-flow}

\subsection{Functional spaces}\label{sec:functional-spaces}

Denote the parabolic cylinder as $Q_r^{\tau_1, \tau_2}(y) \defeq B_r(y) \times [\tau_1, \tau_2]$ and the parabolic annulus as $Q_{r, R}^{\tau_1, \tau_2}(y) \defeq (B_R(y) \setminus B_r(y)) \times [\tau_1, \tau_2]$ for $0 < r < R$ and $\tau_1 < \tau_2$.
When $y = 0$, we omit $y$, and when $\tau_1 = 0$ and $\tau_2 = \infty$ we omit $\tau_1$ and $\tau_2$ for simplicity.
We also denote $Q(y, \tau)$ the unit parabolic cylinder $Q_1^{\tau, \tau + 1}(y)$.

Let $V$ be a Euclidean vector bundle over $\RR^n$. Let $U \subseteq \RR^n \times \RR$ and $f: U \to V$ be a section.
We use the parabolic H\"older norm notation as
\begin{align*}
    [f]_{C^P_{\alpha}; U}      & \defeq \sup \Set{\frac{|f(y_1, \tau_1) - f(y_2, \tau_2)|}{|y_1 - y_2|^\alpha + |\tau_1 - \tau_2|^{\alpha / 2}} \given (y_1, \tau_1), (y_2, \tau_2) \in U, (y_1, \tau_1) \neq (y_2, \tau_2)} \\
    \|f\|_{C^P_{k}; U}         & \defeq \sum_{p + 2q \leq k} \sup_{U} |\nabla_y^p \nabla_\tau^q f|                                                                                                                           \\
    \|f\|_{C^P_{k, \alpha}; U} & \defeq \sum_{p + 2q = k} [\nabla_y^p \nabla_\tau^q f]_{C^P_{\alpha}; U} + \|f\|_{C^P_{k}; U}.
\end{align*}

Let $U' \subseteq \RR^n$, $k \in \NN_0$, $\alpha \in (0, 1)$.
For a section $f: U' \to V$, we use the following notation for the H\"older norm:
\begin{align*}
    \|f\|_{k; U'}         & \defeq \sum_{p \leq k} \sup_{U'} |\nabla^p f|,                                           \\
    [f]_{\alpha; U'}      & \defeq \sup_{y_1, y_2 \in U', y_1 \neq y_2} \frac{|f(y_1) - f(y_2)|}{|y_1 - y_2|^\alpha} \\
    \|f\|_{k, \alpha; U'} & \defeq [\nabla^k f]_{\alpha; U'} + \|f\|_{k; U'}.
\end{align*}

Let $\tilde{r}: \RR^n \to \RR; x \mapsto \zeta(|x|)$ be a scalar function, where $\zeta: [0, \infty] \to [1, \infty]$ is a non-decreasing function with $\zeta(0) = 1$ and $\zeta(r) = r$ for $r \geq 2$ so that $\tilde{r}$ is a smooth function.

We define a weighted H\"older norm $\|f\|_{k, \alpha; U'}^{(d)}$ by
\begin{align*}
    \|f\|_{k; U'}^{(d)}         & \defeq \sum_{p \leq k} \sup_{U'} \tilde{r}^{-d+p} |\nabla^p f|,                                                                                              \\
    [f]_{\alpha; U'}^{(d)}      & \defeq \sup_{y_1, y_2 \in U', y_1 \neq y_2} \frac{1}{\tilde{r}(y_1)^{d - \alpha} + \tilde{r}(y_2)^{d - \alpha}}\frac{|f(y_1) - f(y_2)|}{|y_1 - y_2|^\alpha}, \\
    \|f\|_{k, \alpha; U'}^{(d)} & \defeq [\nabla^k f]^{(d - k)}_{\alpha; U'} + \|f\|_{k; U'}^{(d)}.
\end{align*}
Also, we define another weighted H\"older norm as
\begin{align*}
    \|f\|_{k, \alpha; U'}^{[d]} & \defeq \sup_{y \in \RR^n} \tilde{r}(y)^{-d} \|f\|_{k, \alpha; U' \cap B_1(y)}.
\end{align*}
We omit the domain $U'$ when it is $\RR^n$.

The weighted parabolic Hölder norm with exponential decay is defined as
\begin{align*}
    \|f\|_{k, \alpha; U}^{[d], \delta} & \defeq \sup_{y \in \RR^n, \tau \in \RR} e^{\delta \tau} \tilde{r}(y)^{-d} \|f\|_{C^P_{k, \alpha}; U \cap (Q(y, \tau))}.
\end{align*}
We omit a superscript $[d]$ if $d = 0$, and omit $\delta$ if $\delta = 0$ in the weighted H\"older norm notation.
Note that $\|\cdot\|_{k, \alpha; U}$ and $\|\cdot\|_{C^P_{k, \alpha}; U}$ are equivalent norms.

The following norm will be used to prove the convergence of the flow:
\begin{align*}
    \|u\|_* \defeq \|u\|^{[\gamma + \alpha], \delta}_{2, \alpha} + \|u\|_0^{[-1 + \gamma], \delta} + \|Du\|_0^{[-1 + \gamma], \delta},
\end{align*}
where $0 < \alpha < \gamma < 1/100$ are constants, and $\delta>0$ is a small constant to be determined later.
Here, the operator $D$ denotes the derivative with respect to the spatial variable only.
The last two supnorm terms are included to control the spatial growth of the nonlinear terms in Lemma \ref{lem:nonlinear-estimate}.

\subsection{Weak sense estimates} \label{subsec:weak-sense-estimates}
Throughout this section, we study solutions of the following inhomogeneous linear PDE
\begin{align} \label{eq:inhomogeneous-linear-PDE}
    (\frac{\partial}{\partial \tau} - L)u = h
\end{align}
for $u, h \in \Omega^1(\mf{so}(n)) \times \RR_+$, where $L$ is the linear operator defined in \eqref{eq:de-turck-linearized-operator} and $\RR_+ = [0, \infty)$ in all that follows.
Indeed, we find the solution $u$ with an inhomogeneous term $h = \mc N(u)$ in Section \ref{sec:main-theorem}.

We define the Gaussian-weighed Sobolev spaces
\begin{align*}
    H_{\boldsymbol{\rho}}^k(\RR^n, \Omega^1(\mf{so}(n))) \defeq \Set{u \in \Omega^1(\mf{so}(n)) \given \sum_{p = 0}^{k}\|\nabla^p u\|_{\boldsymbol{\rho}} < \infty},
\end{align*}
which are Hilbert spaces with the inner product and norm
\begin{align*}
    \inn{u, v}_{\boldsymbol{\rho}, k} = \sum_{p = 0}^{k}\inn{\nabla^p u, \nabla^p v}_{\boldsymbol{\rho}}, \quad \|u\|_{\boldsymbol{\rho}, k} = \inn{u, u}_{\boldsymbol{\rho}, k}^{1/2}.
\end{align*}

We can explicitly construct the solution of the inhomogeneous linear PDE \eqref{eq:inhomogeneous-linear-PDE} by the following lemma:
\begin{lemma} \label{lem:existence-unique-sol-parabolic-PDE}
    Fix $\delta > 0$, $0 < \delta' < \min\{\delta, \lambda_{I+1}\}$.
    Suppose that
    \begin{align} \label{eq:integral-assumption-h}
        \int_0^\infty |e^{\delta \tau} \|h(\cdot, \tau)\|_{\boldsymbol{\rho}}|^2 d\tau < \infty.
    \end{align}
    There exists a unique strong $L^2_{\boldsymbol{\rho}}$ solution $u$ of \eqref{eq:inhomogeneous-linear-PDE} such that
    \begin{align}
        \Pi_{>0}(u(\cdot, 0))                                                                                         & = 0,\notag                                                                                \\
        \int_0^\infty |e^{\delta' \tau}(\|u (\cdot, \tau)\|_{\boldsymbol{\rho}, 2} + \|\frac{\partial}{\partial \tau} & u(\cdot, \tau)\|_{\boldsymbol{\rho}})|^2 d\tau  < \infty. \label{eq:ddu-dudt-L2-estimate}
    \end{align}
    The solution is explicitly given by
    \begin{align*}
        u(y, \tau) = \sum_{j = 1}^\infty u^j(\tau) \xi_j(y),
    \end{align*}
    where
    \begin{align}
        h(y, \tau) & \eqdef \sum_{j = 1}^\infty h^j(\tau) \xi_j(y), \notag                                                                     \\
        u^j(\tau)  & \defeq \begin{cases}
                                \int_0^\tau e^{\lambda_j (\sigma - \tau)}h^j(\sigma)d\sigma,      & j \geq I+1,                                     \\
                                -\int_\tau^\infty e^{\lambda_j(\sigma - \tau)}h^j(\sigma)d\sigma, & j = 1, 2, \cdots, I. \label{eq:u-j-for-j-leq-I}
                            \end{cases}
    \end{align}
    Moreover, for all $\tau \in \RR_+$,
    \begin{align}
        e^{\delta' \tau}\|u(\cdot, \tau) \|_{\boldsymbol{\rho}} \leq C \left[\int_0^\infty |e^{\delta \sigma}\|h(\cdot, \sigma)\|_{\boldsymbol{\rho}}|^2 d\sigma\right]^{1/2} \label{eq:u-L2-decay}
    \end{align}
    where $C = C(\delta, \delta', \lambda_{I+1})$.
    \begin{proof}
        First, we show that $u$ is the unique strong solution satisfying \eqref{eq:ddu-dudt-L2-estimate} with the initial condition $u(0) = \sum_{j = 1}^{\infty} u^j(0)\xi_j$.
        H\"older inequality and \eqref{eq:integral-assumption-h} imply
        \begin{align}
            \lambda_j u^j(\tau)^2 & \leq \lambda_j \left(\int_0^\tau e^{\lambda_j(\sigma-\tau)}h^j(\sigma)d\sigma\right)^2                   \notag                                                                   \\
                                  & \leq \lambda_j e^{-2\lambda_j \tau}\left(\int_0^\tau e^{2\delta' \sigma} h^j(\sigma)^2d\sigma\right)^2 \left(\int_0^\tau e^{2(\lambda_j - \delta')\sigma} d\sigma\right)^2 \notag \\
                                  & \leq C(\delta', \lambda_{I+1}) e^{-2\delta' \tau}\left(\int_0^\infty e^{2\delta'\sigma}h^j(\sigma)^2\right) \label{eq:labmda-u-u-decay-positive}
        \end{align}
        for $j \geq I+1$ and $\tau \in \RR_+$.
        Similarly,
        \begin{align}
            u^j(\tau)^2 & \leq \left(\int_\tau^\infty e^{\lambda_j(\sigma-\tau)}h^j(\sigma)d\sigma\right)^2   \notag                                                                                         \\
                        & \leq e^{-2\lambda_j \tau}\left(\int_\tau^\infty e^{2\delta' \sigma} h^j(\sigma)^2 d\sigma\right)^2 \left(\int_\tau^\infty e^{2(\lambda_j - \delta')\sigma} d\sigma\right)^2 \notag \\
                        & \leq C(\delta') e^{-2\delta' \tau}\left(\int_0^\infty e^{2\delta'\sigma}h^j(\sigma)^2\right) \label{eq:labmda-u-u-decay-negative}
        \end{align}
        for $j \leq I$ and $\tau \in \RR_+$.
        Combining \eqref{eq:labmda-u-u-decay-positive}, \eqref{eq:labmda-u-u-decay-negative} and \eqref{eq:integral-assumption-h}, we deduce for all $\tau \in \RR_+$,
        \begin{align*}
            \|u(\cdot, \tau)\|_{\boldsymbol{\rho}} & \leq [\sum_{j = 1}^{\infty} u^j(\tau)^2]^{1/2} \leq C(\delta', \lambda_{I+1})e^{-\delta'\tau} \left[ \int_0^\infty e^{2\delta' \sigma}\|h(\cdot, \sigma)\|_{\boldsymbol{\rho}}^2 d\sigma \right]^{1/2} \\
                                                   & \leq C(\delta', \lambda_{I+1}) e^{-\delta' \tau}\left[\int_0^\infty|e^{\delta\sigma}\|h(\cdot, \sigma)\||^2\right]^{1/2},
        \end{align*}
        and get the inequality \eqref{eq:u-L2-decay}.
        Since $\|u(\cdot, 0)\| < \infty$, we deduce that $u = \sum_{j = 1}^{\infty} u^j \xi_j$ is a unique weak $L_{\boldsymbol{\rho}}^2$ solution of \eqref{eq:inhomogeneous-linear-PDE} with an initial condition $u(\cdot, 0) = \sum_{j=1}^{\infty} u^j(0)\xi_i$ by adaptation of Galerkin's method (\cite[\S 7.1.2, Theorems 3, 4]{Evans}).
        Combining \eqref{eq:labmda-u-u-decay-negative} and \eqref{eq:labmda-u-u-decay-positive} gives
        \begin{align} \label{eq:lambda-uu-decay}
            \sum_{j = 1}^{\infty}\lambda_j u^j(\tau)^2 \leq C(\delta', \lambda_{I+1}) e^{-2\delta' \tau}.
        \end{align}
        We introduce simple inequality
        \begin{claim}
            Let $w \defeq \sum_{j = 1}^{\infty} w^j \xi_j \in L^2_{\boldsymbol{\rho}}(\Omega^1(\mf{so}(n)))$.
            Then,
            \begin{align*}
                \sum_{j=1}^\infty \lambda_j (w^j)^2 < \infty \iff w \in H_{\boldsymbol{\rho}}^1(\Omega^1(\mf{so}(n))).
            \end{align*}
            Furthermore,
            \begin{align*}
                \frac{1}{2}\|w\|_{\boldsymbol{\rho}, 1}^2 - C\|w\|_{\boldsymbol{\rho}}^2 \leq \sum_{j = 1}^{\infty} \lambda_j (w^j)^2 \leq 2\|w\|_{\boldsymbol{\rho}, 1} + C\|w\|_{\boldsymbol{\rho}}.
            \end{align*}
            \begin{proof}
                By H\"older inequality and \eqref{eq:de-turck-linearized-operator}, we have
                \begin{align*}
                    \sum_{j = 1}^{\infty} \lambda_j (u^j(\tau))^2 & = \| \nabla u \|^2_{\boldsymbol{\rho}} + \langle u_j, \frac{1}{2}u_j - 2[W_i, \partial_i u_j] + [W_i, [W_i, u_j]] + 2[u_i, {F_W}_{ij}] \rangle_{\boldsymbol{\rho}} \\
                                                                  & \geq \frac{1}{2}\|\nabla u\|^2_{\boldsymbol{\rho}, 1} - C\|u\|^2_{\boldsymbol{\rho}}.
                \end{align*}
                The upper bound can be derived similarly.
            \end{proof}
        \end{claim}
        Consequently, \eqref{eq:lambda-uu-decay} and the claim imply
        \begin{align*}
            \|u(\cdot, \tau)\|_{\boldsymbol{\rho}, 1} \leq C(\delta', \lambda_{I+1}) e^{-\delta' \tau}.
        \end{align*}
        Also, we have
        \begin{align*}
            \|\frac{\partial u}{\partial \tau}\|_{\boldsymbol{\rho}, -1} \leq C e^{-\delta'\tau}
        \end{align*}
        for a.e. $\tau \in \RR_+$, since
        \begin{align*}
            \langle \frac{\partial u}{\partial \tau}(\tau), v \rangle_{\boldsymbol{\rho}} & = \langle \sum_{j = 1}^{\infty}(h^j - \lambda_j u^j)(\tau)\xi_j, \sum_{j=1}^{\infty} v^j \xi_j\rangle_{\boldsymbol{\rho}} = \sum_{j = 1}^\infty h^j(\tau)v^j -\lambda_j u^j(\tau)v^j   \\
                                                                                          & \leq \|h\|_{\boldsymbol{\rho}} \|v\|_{\boldsymbol{\rho}}+ \left(\sum_{j = 1}^{\infty} |\lambda_j| u^j(\tau)^2\right)^{1/2} \left(\sum_{j = 1}^{\infty} |\lambda_j|(v^j)^2\right)^{1/2} \\
                                                                                          & \leq \|h\|_{\boldsymbol{\rho}} + C\|u\|_{\boldsymbol{\rho}, 1}\|v\|_{\boldsymbol{\rho}, 1} \leq C e^{-\delta' \tau}.
        \end{align*}
        for any $v = \sum_{i = 1}^\infty v^i \xi_i \in H_{\boldsymbol{\rho}}^1(\Omega^1(\mf{so}(n)))$ by H\"older inequality and the claim.
        Indeed, $u$ is a strong solution with the following estimate (\cite[\S 7.1.3, Theorem 5, Eq (46)]{Evans})
        \begin{align*}
            \int_{\tau}^{\tau+1}(\|u(\cdot, \tau)\|_{\boldsymbol{\rho}, 2} + \|\frac{\partial u}{\partial \tau}(\cdot, \tau)\|_{\boldsymbol{\rho}})^2 d \sigma \leq C (\int_{\tau}^{\tau + 1}\|h\|^2_{\boldsymbol{\rho}} d\sigma + \|u(\cdot, \tau)\|^2_{\boldsymbol{\rho}, 1})
        \end{align*}
        for a.e. $\tau \in \RR_+$, which implies \eqref{eq:ddu-dudt-L2-estimate}.
        In the similar way, for any $\mathbf{a} \in \RR^{I}$, we can show that $u[\mathbf{a}] = \sum_{j=1}^{\infty}u[\mathbf{a}]^j \xi_j$ is a unique solution of \eqref{eq:inhomogeneous-linear-PDE} with the initial condition $(u[\mathbf{a}])(\cdot, 0) = \sum_{j = 1}^{I} a^j\xi_j$, where
        \begin{align*}
            u[\mathbf{a}]^j(\tau) = \begin{cases}
                                        u^j(\tau) - u^j(0) + \mathbf{a}^j, & j \geq I+1,          \\
                                        u^j(\tau) d\sigma,                 & j = 1, 2, \cdots, I.
                                    \end{cases}
        \end{align*}
        Hence, a condition $\Pi_{>0}(u(\cdot, 0)) = 0$ and an inequality \eqref{eq:ddu-dudt-L2-estimate} determine $\mathbf{a}$ to be $(u^1(0), \cdots, u^{I}(0))$ and this completes the proof of uniqueness part.
    \end{proof}
\end{lemma}

We introduce a variant of Kato's inequality for the Frobenius norm.
\begin{lemma}[Kato's inequality for the Frobenius norm]\label{lem:Kato-inequality}
    Let $u \in H^2_{\boldsymbol{\rho}}(\Omega^1(\mf{so}(n)))$.
    Then, the following holds
    \begin{align} \label{eq:Kato-inequality}
        \Delta \|u\|_F \geq \left(\frac{\langle u, \Delta u\rangle_F}{\|u\|_F} + \frac{1}{n} \frac{\sum_{i,j}\|[\partial_i u_j, u_j]\|_F^2}{\|u\|_F^3}\right) \chi_{u \neq 0},
    \end{align}
    in the weak $L^2_{\boldsymbol{\rho}}$ sense.
    \begin{proof}[Proof of Lemma \ref{lem:Kato-inequality}]
        First, we show the simple inequality
        \begin{align} \label{eq:final-matrix-inequality}
            \sum_i \langle \partial_i u, \partial_i u \rangle_F \|u\|_F^2 - \langle \partial_i u, u\rangle_F^2 \geq \frac{1}{n}\sum_{i,j} \|[\partial_i u_j, u_j]\|_F^2.
        \end{align}
        The following matrix inequality for $A, B \in \text{Mat}_n(\RR)$
        \begin{align}
            \|[A, B]\|_F^2 & = \sum_{\mu, \nu}(\sum_\eta A^\mu_\eta B^\eta_\nu - B^\mu_\eta A^\eta_\nu)^2 \leq n\sum_{\mu, \nu, \eta} (A^\mu_\eta B^\eta_\nu - B^\mu_\eta A^\eta_\nu)^2 \notag \\
                           & \leq n\sum_{\mu, \nu, \eta, \xi} (A^\mu_\eta B^\xi_\nu - B^\mu_\eta A^\xi_\nu)^2 = n(\|A\|_F^2\|B\|_F^2 - \langle A, B\rangle_F^2), \label{eq:matrix-inequality}
        \end{align}
        implies
        \begin{align*}
            \|u_j\|_F^2 \|\partial_i u_j\|^2_F - \langle u_j, \partial_i u_j\rangle_F^2 \geq \frac{1}{n} \|[\partial_i u_j, u_j]\|_F^2.
        \end{align*}
        Note also
        \begin{align} \label{eq:matrix-inequality-2}
            \|u_j\|^2_F \|\partial_i u_k\|^2_F + \|u_k\|_F^2\|\partial_i u_j\|_F^2 \geq 2 \langle u_j, \partial_i u_j \rangle_F \langle u_k, \partial_i u_k\rangle_F
        \end{align}
        by the AM-GM and Cauchy-Schwarz inequality.
        Combining \eqref{eq:matrix-inequality} and \eqref{eq:matrix-inequality-2} gives
        \begin{align*}
            (\sum_j \|u_j\|_F^2)(\sum_j \|\partial_i u_j\|_F^2) - (\sum_j\langle u_j, \partial_i u_j\rangle_F)^2 \geq \frac{1}{n}\sum_j \|[\partial_i u_j, u_j]\|_F^2
        \end{align*}
        or equivalently \eqref{eq:final-matrix-inequality}.
        Let $u_\epsilon = \sqrt{\|u\|_F^2 + \epsilon^2} \in H^1_{\boldsymbol{\rho}}(\Omega^1(\mf{so}(n)))$.
        Let $\phi \in C^\infty_c(\RR^n)$ be a non-negative test function.
        Integration by parts and \eqref{eq:final-matrix-inequality} gives
        \begin{align}
             & \int (\nabla u_\epsilon \cdot \nabla \phi - \frac{1}{2}(y \cdot \nabla u_\epsilon) \phi) \boldsymbol{\rho}                                                                                                                                                                             \notag  \\
             & = \int \sum_i \left(\langle \partial_i u, \partial_i(\frac{u}{u_\epsilon}\phi) - \frac{y^i}{2}(\frac{u}{u_\epsilon}\phi) \rangle_F - \frac{\langle \partial_i u, \partial_i u \rangle_F u_\epsilon^2 - \langle \partial_i u, u \rangle_F^2}{u_\epsilon^3} \phi\right) \boldsymbol{\rho} \notag \\
             & \leq \int (\langle \Delta u, \frac{u}{u_\epsilon}\phi\rangle_F - \frac{1}{n} \frac{\sum_{i,j} \|[\partial_i u_j, u_j]\|_F^2}{u_\epsilon^3} \phi) \boldsymbol{\rho}. \label{eq:weak-Kato-inequality}
        \end{align}
        Since $u_\epsilon \to u$ in $H^1_{\boldsymbol{\rho}}$ and $\frac{u}{u_\epsilon} \phi \to \chi_{u \neq 0}\phi$ in $L^2_{\boldsymbol{\rho}}$ as $\epsilon \to 0$, all terms in the inequality \eqref{eq:weak-Kato-inequality} converges except for the last term.
        Indeed, monotonicity with respect to $\epsilon$ gives the convergence of the last term.
        Taking $\epsilon \to 0$ gives the desired inequality \eqref{eq:Kato-inequality}.
    \end{proof}
\end{lemma}
For simplicity, we omit $\chi_{u \neq 0}$ in the following.
We establish a sup-norm estimate for the solution of the inhomogeneous linear PDE \eqref{eq:inhomogeneous-linear-PDE} in the following lemma.

\begin{lemma} \label{lem:supnorm-estimate}
    Suppose $u, h$ satisfies \eqref{eq:inhomogeneous-linear-PDE} and
    \begin{align*}
        \int_{\tau}^{\tau+1} \left(\|u(\cdot, \sigma)\|_{\boldsymbol{\rho}, 2}^2 + \|\frac{\partial}{\partial \tau} u(\cdot, \sigma)\|_{\boldsymbol{\rho}}^2\right) d\sigma < \infty
    \end{align*}
    for all $\tau \in \RR_+$.
    Then, the following estimates hold.
    \begin{enumerate} [label=(\alph*)]
        \item \label{item:exp-decay-supnorm-estimate}
              Let $d > -1$ and $0 < \delta < \frac{1+d}{2}$.
              For all $\tau \in \RR_+$, and $R > R_0(n, d, \delta)$,
              \begin{align*}
                  \|u\|_{0}^{[d], \delta} \leq C(n, d, \delta) \left(\|u\|_{0 ; Q_R}^{[d],\delta} + \|h\|_0^{[d], \delta} + \|u(\cdot, 0)\|^{[d]}_{0}\right).
              \end{align*}
        \item \label{item:-1-supnorm-estimate}
              Let $d' < -1$.
              For all $\tau \in \RR_+$, and $R > R_0(n, d')$,
              \begin{align}
                  \|u\|_{0}^{[-1]} \leq C(n, d') \left(\|u\|_{0 ; Q_R}^{[-1]} + \|h\|_{0}^{[d']} + \|u(\cdot, 0)\|^{[-1]}_{0}\right) \label{eq:-1-supnorm-estimate}.
              \end{align}
    \end{enumerate}
    \begin{proof}
        We follow the proof of \cite[Lemma 3.15]{CHODOSH}, but use Lemma \ref{lem:Kato-inequality} instead of the standard Kato's inequality.
        Since the inequality \eqref{eq:-1-supnorm-estimate} holds for a particular value $d' = d_0'$, it must also hold for all $d'$ less than $d_0'$.
        Therefore, we can assume that $d' \in (-3, -1)$.
        Hence, we assume that $-3 < d' < -1$.
        Let $\langle \cdot, \cdot \rangle$ be an inner product of $\mf{so}(n)$-valued tensors, defined as follows:
        \begin{align*}
            \langle T, V \rangle_F \defeq \sum_{\alpha \in I} \text{tr}(T_{\alpha}^t V_{\alpha}),
        \end{align*}
        where $I$ is the index set of $T, V$, $T^t$ is a transpose of $T$ and $\text{tr}$ is a trace of a matrix.
        We also denote the corresponding norm as $\|T\|_F = \inn{T, T}_F^{1/2}$.
        We aim to bound $\|u\|_F^2$ on $\RR^n \setminus B_R$ using a barrier function with a primary decay $\sim|y|^{d}$.

        Using Lemma \ref{lem:Kato-inequality} and \eqref{eq:de-turck-linearized-operator}, we have
        \begin{align}
             & (\cfrac{\partial}{\partial \tau} - (\Delta - \frac{1}{2} y \cdot \nabla - \frac{1}{2}))\|u\|_F                                                                                                                  \notag \\
             & \leq \frac{\langle u, (\partial_\tau - (\Delta - \frac{1}{2} y \cdot \nabla - \frac{1}{2}))u\rangle_F}{\|u\|_F} - \frac{\frac{1}{n}\sum_{i,j}\|[\partial_i u_j, u_j]\|_F^2}{\|u\|_F^3}                          \notag \\
             & \leq \frac{\sum_{i,j} \left(\langle u_j, h_j + 2[W_i, \partial_i u_j] + [W_i, [W_i, u_j]] + 2[u_i, {F_W}_{ij}]\rangle_F\right)}{\|u\|_F} - \frac{\frac{1}{n}\sum_{i,j}\|[\partial_i u_j, u_j]\|_F^2}{\|u\|_F^3} \notag \\
             & \leq \frac{\langle u, h + \frac{C}{\tilde{r}^2}u\rangle_F}{\|u\|_F} +  \sum_{i, j}(\frac{2\langle W_i, [\partial_i u_j, u_j] \rangle_F}{\|u\|_F} - \frac{\frac{1}{n}\|[\partial_i u_j, u_j]\|_F^2}{\|u\|_F^3})  \notag \\
             & \leq \frac{\langle u, h + \frac{C}{\tilde{r}^2}u \rangle_F}{\|u\|_F} + \|u\|_F\frac{C}{\tilde{r}^2} \leq \|h\|_F + \frac{C_0}{\tilde{r}^2} \|u\|_F\label{eq:dt-L-estimate}
        \end{align}
        in the weak sense for some $C_0(n) > 0$.
        By defining the linear operator
        \begin{align*}
            L' \defeq \Delta - \frac{1}{2}y \cdot \nabla - \frac{1}{2} + \frac{C_0}{r^2}
        \end{align*}
        on $\RR^n \setminus B_R$, we obtain
        \begin{align*}
            (\frac{\partial}{\partial \tau} - L') \|u\|_F \leq \|h\|_F.
        \end{align*}
        in the weak sense.
        Let $f = \|u\|_F - \psi$, where $\psi$ is a barrier function defined differently for each case (a) and (b).
        We will show the following estimates for each case, assuming $R$ is sufficiently large:
        \begin{align*}
            f                                      & < 0 \text{ on } (\partial B_R \times \RR_+) \cup ((\RR^n\setminus B_R) \times \Set{0}) \\
            (\frac{\partial}{\partial \tau} - L')f & \leq 0 \text{ in } Q_{R, \infty}.
        \end{align*}
        \begin{enumerate} [label=(\alph*)]
            \item Define the barrier function $\psi$ by
                  \begin{align*}
                      \psi(y, \tau) = \alpha r^{d}e^{-\delta \tau}
                  \end{align*}
                  with
                  \begin{align*}
                      \alpha & = 2 \left(\|u(\cdot, 0)\|^{[d]}_0 + \sup_{\partial B_R \times \RR_+} \left[e^{\delta \tau} \|r^{-d} u\|_F\right]\right) + \frac{4}{d+1 - 2\delta} \|h\|_{0}^{[d], \delta}.
                  \end{align*}
                  If $\alpha = 0$, then \eqref{eq:-1-supnorm-estimate} directly holds.
                  Therefore, we can assume $\phi>0$ \ie $\alpha > 0$.
                  This assumption gives strict inequality
                  \begin{align}
                      f & = \|u\|_F - \psi = \|u\|_F - \alpha r^{d}e^{-\delta \tau}                                                                                    \notag \\
                        & \leq \|u\|_F - 2r^d e^{-\delta \tau}(\|u(\cdot, 0)\|^{(d)}_0 + \sup_{\partial B_R \times \RR_+} \left[e^{\delta \tau} \|r^{-d} u\|_F\right]) \notag \\
                        & < 0 \text{ on } (\partial B_R \times \RR_+) \cup ((\RR^n \setminus B_R) \times \Set{0}).\label{eq:barrier-boundary-inequality}
                  \end{align}
                  On the other hand, for sufficiently large $R > R(n, d, \delta)$, we have
                  \begin{align}
                      (\frac{\partial}{\partial \tau} - L')f & \leq \|h\|_F - (\frac{\partial}{\partial \tau} - L')\psi                                                                                  \notag \\
                                                             & = \|h\|_F - \alpha ((\frac{1 + d}{2} - \delta)r^{d} - (C_0 + d(n + d - 2))r^{d-2})e^{-\delta \tau}                                        \notag \\
                                                             & \leq \|h\|_F - \frac{1}{2}\alpha(\frac{1 + d}{2} - \delta) r^d e^{-\delta \tau} \leq \|h\|_F - r^d e^{-\delta \tau} \|h\|_0^{[d], \delta} \notag \\
                                                             & \leq 0 \text{ in } Q_{R, \infty}.\label{eq:barrier-L'-inequality}
                  \end{align}
            \item Define the barrier function $\psi$ by
                  \begin{align*}
                      \psi(y, \tau) = \alpha'(\frac{1}{r} - r^{d'})
                  \end{align*}
                  with
                  \begin{align*}
                      \alpha' & = 2 \left(\|u(\cdot, 0)\|^{(-1)}_0 + \sup_{\partial B_R \times \RR_+} \left[\|ru\|_F\right]\right) + \frac{4}{-1-d'} \|h\|_0^{[d']}.
                  \end{align*}
                  We assume that $\alpha' \neq 0$.
                  For $(y, \tau) \in (\partial B_R \times \RR_+) \cup ((\RR^n \setminus B_R) \times \Set{0})$, we have
                  \begin{align*}
                      f(y, \tau) & = \|u\|_F - \alpha'(\frac{1}{r} - r^{d'}) \leq \|u\|_F - \frac{\alpha'}{2r}                                          \\
                                 & \leq \|u\|_F - \frac{1}{r}(\|u(\cdot, 0)\|^{(-1)}_0 + \sup_{\partial B_R \times \RR_+} \left[\|ru\|_F\right]) \leq 0
                  \end{align*}
                  for sufficiently large $R>R(d')$.
                  By an assumption $-3 < d' <-1$, we obtain
                  \begin{align*}
                      (\frac{\partial}{\partial \tau} - L')f & \leq \|h\|_F - (\frac{\partial}{\partial \tau} - L')\psi                                   \\
                                                             & = \|h\|_F + \alpha' (C_0 -(n - 3))r^{-3}                                                   \\
                                                             & - \alpha' ((\frac{-1 - d'}{2}r^{d'} + (C_0 + d'(n + d' - 2)))r^{d'-2})                     \\
                                                             & \leq \|h\|_F - \frac{-1-d'}{4}\alpha' r^{d'} \leq \|h\|_F - r^{d'} \|h\|_0^{[d']} \leq 0 .
                  \end{align*}
                  in $Q_{R, \infty}$ for sufficiently large $R > R(n, d')$.
        \end{enumerate}

        Define a function $f_+ \defeq \max\{f, 0\} \in L^2(\tau, \tau + 1; H_{\boldsymbol{\rho}}^1(\RR^n \setminus B_R))$ that vanishes in some neighborhood of the boundary due to \eqref{eq:barrier-boundary-inequality}.
        Since $\frac{\partial f_+}{\partial \tau} = \frac{\partial u}{\partial \tau}\chi_{f>0} \in L^2(\tau, \tau+1; L^2_{\boldsymbol{\rho}}(\RR^n \setminus B_R))$, \cite[\S 5.9.2, Theorem 3]{Evans} applied to $f_+$ gives absolute continuity of $\|f_+(\cdot, \tau)\|_{\boldsymbol{\rho}; \RR^n \setminus B_R}$ with
        \begin{align}
            \frac{1}{2}\frac{\partial}{\partial \tau} \|f_+\|^2_{\boldsymbol{\rho}; \RR^n \setminus B_R} = \langle \frac{\partial}{\partial \tau} f_+, f_+\rangle_{\boldsymbol{\rho}; \RR^n \setminus B_R} = \langle \frac{\partial}{\partial \tau}f, f_+ \rangle_{\boldsymbol{\rho}; \RR^n \setminus B_R} \label{eq:barrier-tau-derivative-inequality}
        \end{align}
        for a.e. $\tau \in \RR_+$.
        Multiply \eqref{eq:barrier-L'-inequality} by $f_+$ and integrate over $\RR^n \setminus B_R$ yields
        \begin{align}
            \frac{1}{2}\frac{\partial}{\partial \tau} \|f_+\|^2_{\boldsymbol{\rho}; \RR^n \setminus B_R} & \leq \int_{\RR^n \setminus B_R} f_+ (\Delta - \frac{1}{2}y \cdot \nabla - \frac{1}{2} + \frac{C_0}{r^2})f \boldsymbol{\rho}                                                                                       \notag \\
                                                                                                         & = \int_{\RR^n \setminus B_R} -|\nabla f_+|^2 \boldsymbol{\rho} - (\frac{1}{2} - \frac{C_0}{R^2})\|f_+\|_{\boldsymbol{\rho}; \RR^n \setminus B_R}^2 \leq 0 \label{eq:barrier-tau-derivative-inequality-integrated}
        \end{align}
        by \eqref{eq:barrier-tau-derivative-inequality} and integration by parts for a.e. $\tau \in \RR_+$ and sufficiently large $R$.
        Since \eqref{eq:barrier-boundary-inequality} implies $\|f_+(\cdot, 0)\|_{\boldsymbol{\rho}; \RR^n \setminus B_R} = 0$, we conclude that $f_+ \equiv 0$ in $Q_{R, \infty}$.
    \end{proof}
\end{lemma}

\subsection{Pointwise estimates} \label{subsec:pointwise-estimates}
We have the following interior estimate:
\begin{lemma} \label{lem:weighted-holder-estimate-on-cpt}
    Let $u, h$ be solutions to \eqref{eq:inhomogeneous-linear-PDE}.
    Suppose
    \begin{align*}
        \|h\|_{0, \alpha}^{\bar{\delta}} < \infty, \quad \Pi_{>0}(u(\cdot, 0)) \equiv 0,
    \end{align*}
    and
    \begin{align*}
        \int_0^\infty |e^{\delta'\tau} (\|u(\cdot, \tau)\|_{\boldsymbol{\rho}, 2} + \|\frac{\partial}{\partial \tau} u(\cdot, \tau)\|_{\boldsymbol{\rho}})|^2 d\tau < \infty.
    \end{align*}
    where $0 < \delta' < \min\{\bar{\delta}, \lambda_{I+1}\}$ and $d \in \RR$.
    Then, for every $R>1$, we have
    \begin{align*}
        \|u\|_{2, \alpha; Q_R}^{\delta'} \leq C(\alpha, \delta', \bar{\delta}, R) \left(\|h\|_{0, \alpha}^{\bar{\delta}} + \|u(\cdot, 0)\|_{2, \alpha; B_{R+1}}\right).
    \end{align*}
    \begin{proof}
        Note that the coefficients of $L$ have the bounded $\|\cdot\|_{0, \alpha; B_{R+1} \times [0, \infty)}$-norm.
        Applying Lemma \ref{lem:existence-unique-sol-parabolic-PDE} with $\delta = (\delta' + \bar{\delta}) / 2$ implies
        \begin{align*}
            e^{\delta' \tau} \|u(\cdot, \tau)\|_{\boldsymbol{\rho}} & \leq C(\delta', \bar{\delta}) \left[ \int_0^\infty |e^{\delta \tau} \|h(\cdot, \tau)\|_{\boldsymbol{\rho}}|^2 d\tau \right]^{1/2} \notag \\
                                                                    & \leq C(\delta', \bar{\delta}) \|h\|_{0, \alpha}^{\bar{\delta}}.
        \end{align*}
        Then the interior $L^2$-Schauder estimate(Corollary \ref{thm:parabolic-Schauder-Lp-estimate}) to a ball cover of $B_R$ gives the exponential decay result
        \begin{align}
             & \|u\|_{2, \alpha; Q_R^{\tau-1/2, \tau}} \leq C(R, \alpha) \left(\|h\|_{0, \alpha; Q_{R+1}^{\tau-1, \tau}} + \|u\|_{L^2(Q_{R+1}^{\tau-1, \tau})}\right) \notag \\
             & \leq C(R, \alpha, \delta', \bar{\delta}) \left(\|h\|_{0, \alpha ; Q_{R+1}^{\tau-1, \tau}} + e^{-\delta' \tau} \|h\|_{0, \alpha}^{\bar{\delta}}\right)  \notag \\
             & \leq C(R, \alpha, \delta', \bar{\delta}) e^{-\delta' \tau} \|h\|_{0, \alpha}^{\bar{\delta}}\label{eq:cpt-weighted-holder-estimate-on-final}
        \end{align}
        for all $\tau \in [1, \infty)$.
        Similarly, Corollary \ref{cor:parabolic-Schauder-Lp-initial-data-estimate} gives the estimate near the initial time:
        \begin{align}
            \|u\|_{2, \alpha; Q_R^{0, 1}} & \leq C(R, \alpha) \left(\|u\|_{L^2(Q_{R+1}^{0, 1})} + \|h\|_{0, \alpha; Q_{R+1}^{0, 1}} + \|u(\cdot, 0)\|_{2, \alpha; B_{R+1}}\right) \notag                                    \\
                                          & \leq C(R, \alpha, \delta', \bar{\delta})\left(\|h\|_{0, \alpha}^{\bar{\delta}} + \|u(\cdot, 0)\|_{2, \alpha; B_{R+1}}\right).\label{eq:cpt-weighted-holder-estimate-on-initial}
        \end{align}
        Combining \eqref{eq:cpt-weighted-holder-estimate-on-initial} and \eqref{eq:cpt-weighted-holder-estimate-on-final} gives the desired estimate.
    \end{proof}
\end{lemma}
\begin{lemma} \label{lem:our-holder-estimate}
    Suppose $u, h$ satisfy \eqref{eq:inhomogeneous-linear-PDE}.
    Let $d \in \RR$ and $\delta > 0$.
    Then, we have the following estimates.
    \begin{enumerate} [label=(\alph*)]
        \item \label{item:our-spatial-Schauder-estimate}
              Spatial H\"older norm estimate is given by
              \begin{align*}
                  \sup_{\tau \in [0, \infty)} \left[e^{\delta \tau} \|u(\cdot, \tau)\|_{2, \alpha}^{[d]}\right] \leq C\left(\|u(\cdot, 0)\|^{[d]}_{2, \alpha} + \|h\|^{[d - \alpha], \delta}_{0, \alpha} + \|u\|_0^{[d], \delta}\right)
              \end{align*}
        \item \label{item:our-parabolic-Schauder-estimate}
              Parabolic H\"older norm estimate is given by
              \begin{align*}
                  \|u\|_{2, \alpha}^{[1 + d + \alpha], \delta} \leq C\left(\|u(\cdot, 0)\|^{[d]}_{2, \alpha} + \|h\|^{[d - \alpha], \delta}_{0, \alpha} + \|u\|_0^{[d], \delta}\right).
              \end{align*}
    \end{enumerate}
    Here, the constant $C$ depends only on $n, \alpha, \delta, d$.
    \begin{proof}
        The coefficients of $L$ have the bounded $\|\cdot\|_{0, \alpha; Q_{20}}$-norm.
        Then, Theorem \ref{thm:parabolic-Schauder-initial-data-estimate} gives
        \begin{align*}
            \|u\|_{2, \alpha; Q_{15}^{0, 1}} \leq C \left( \|u(\cdot, 0)\|_{2, \alpha; B_{20}} + \|u\|_{0; Q_{20}^{0, 1}} + \|h\|_{0, \alpha; Q_{20}^{0, 1}}\right)
        \end{align*}
        and Theorem \ref{thm:parabolic-Schauder-estimate} implies
        \begin{align*}
            \|u\|_{2, \alpha; Q_{15}^{\tau+1/2, \tau + 1}} \leq C (\|u\|_{0; Q_{20}^{\tau, \tau + 1}} + \|h\|_{0, \alpha; Q_{20}^{\tau, \tau + 1}})
        \end{align*}
        for all $\tau \in \RR_+$.
        Combining these estimates yields the desired result for the cylinder $Q_{15}$ as
        \begin{align} \label{eq:our-parabolic-Schauder-estimate-cpt}
            \|u\|_{2, \alpha; Q_{1}^{\tau, \tau+1}(y)} \leq C e^{-\delta \tau}\left( \|u(\cdot, 0)\|_{2, \alpha}^{[d]} + \|h\|_{0, \alpha}^{[d - \alpha, \delta]}  + \|u\|_0^{[d], \delta}\right)
        \end{align}
        for all $(y, \tau) \in Q_{15}$.

        For the remaining region, we rescale the function $u$.
        Let $\tau_0 \in \RR_+$ be arbitrary.
        We work with the rescaled variables
        \begin{align*}
            t = -e^{\tau_0 - \tau} \text{ and } x = e^{(\tau_0 - \tau)/2} y.
        \end{align*}
        From this point forward, we will consistently use $(t, x)$ and $(\tau, y)$ to represent these variables.
        Define the rescaled function $v(x, t)$ as
        \begin{align*}
            v(x, t) \defeq \frac{1}{\sqrt{-t}}u(\frac{x}{\sqrt{-t}}, -\log (-t) + \tau_0)
        \end{align*}
        in the region $t \in [-1, 0)$.
        We introduce a H\"older norm comparison between rescaled functions.
        \begin{claim}
            Let $\xi(x, t) = \eta(\frac{x}{\sqrt{-t}}, -\log(-t) + \tau_0)$.
            For all $R, a > 1$, $k \in \Set{0, 1}$ and $T \in [0, 1/2]$, we have the following estimates:
            \begin{enumerate} [label=(\alph*)]
                \item $\|\xi\|_{2k, \alpha; Q_{R, aR}^{-e^{-T}, -1/e}} \leq C(a, k) R^{k + \alpha} \|\eta\|_{2k, \alpha; Q_{R, \sqrt{e}aR}^{\tau_0 + T, \tau_0 + 1}}$,
                \item $\|\eta\|_{2k, \alpha; Q_{R, aR}^{\tau_0 + T, \tau_0 + 1}} \leq C(a, k) R^{k + \alpha} \|\xi\|_{2k, \alpha; Q_{R/\sqrt{e},aR}^{-e^{-T}, -1/e}}$.
            \end{enumerate}
        \end{claim}
        The proof of the claim follows directly from a straightforward calculation, where the $R^{k + \alpha}$ factor arises due to the time-direction differentiation in rescaled coordinates.
        Denote the rescaled variables $\hat{W}, \hat{F_W}, \hat{h}$ as
        \begin{align}
            \hat{W}_i(x, t) & \defeq \frac{1}{\sqrt{-t}} W_i(\frac{x}{\sqrt{-t}}), \quad \hat{F_W}_{ij} \defeq \frac{1}{-t} {F_W}_{ij}(\frac{x}{\sqrt{-t}}),\label{eq:rescaled-W-F} \\
            \hat{h}(x, t)   & \defeq \frac{1}{\sqrt{-t}^3}h( \frac{x}{\sqrt{-t}}, -\log(-t) + \tau_0). \notag
        \end{align}
        The rescaled flow equation for $v(x, t)$ becomes
        \begin{align*}
            \frac{\partial v(x, t)}{\partial t} = \hat{L}v(x, t) + \hat{h}(x, t)
        \end{align*}
        where
        \begin{align*}
            \hat{L} v_j(x, t) \defeq \Delta_x v_j + 2[\hat{W}_i, \frac{\partial}{\partial x^i} v_j] + [\hat{W}_i, [\hat{W}_i, v_j]] + 2[v_i, \hat{F_W}_{ij}].
        \end{align*}
        The $\|\cdot\|_{0, \alpha; \RR^n \times [-1, -\frac{1}{e}]}$-norm of coefficients of $\hat{L}$ are bounded by a constant $C(R, d)$.

        For the initial time estimate, set $\tau_0 = 0$ and $R > 10$ arbitrarily.
        By applying parabolic Schauder estimate with an initial boundary condition(see Theorem \ref{thm:parabolic-Schauder-initial-data-estimate}) to $Q_{R/4,3R}^{-1, -1/4}$, we obtain
        \begin{align}
            \|v\|_{2, \alpha; Q_{R/2, 2R}^{-1, -1/3}} & \leq C \left( \|v(\cdot, -1)\|_{2, \alpha; B_{3R} \setminus B_{R/4}} + \|\hat{h}\|_{0, \alpha; Q_{R/4,3R}^{-1, -1/3}} + \|v\|_{0; Q_{R/4,3R}^{-1, -1/3}}\right) \notag \\
                                                      & \leq C \left( \|u(\cdot, 0)\|_{2, \alpha; B_{3R} \setminus B_{R/4}} + R^{\alpha}\|h\|_{0, \alpha; Q_{R/4, 6R}^{0, 1}} + \|u\|_{0; Q_{R/4, 6R}^{0, 1}}\right)    \notag \\
                                                      & \leq C R^d \left( \|u(\cdot, 0)\|_{2, \alpha}^{[d]} + \|h\|_{0, \alpha}^{[d-\alpha], \delta} + \|u\|_0^{[d]}\right)\label{eq:our-parabolic-Schauder-estimate-initial}
        \end{align}
        where we used the claim for the second inequality.

        Now, let $\tau_0 \geq 0$ be arbitrary.
        The parabolic Schauder estimate(see Theorem \ref{thm:parabolic-Schauder-estimate}) applied to $Q_{R/4, 3R}^{\tau_0 + 1/2, \tau_0 + 1}$ gives
        \begin{align}
            \|v\|_{2, \alpha; Q_{R/2, 2R}^{-1/2, -1/3}} & \leq C \left( \|\hat{h}\|_{0, \alpha; Q_{R/4, 3R}^{-1, -1/3}} + \|v\|_{0; Q_{R/4, 3R}^{-1, -1/3}}\right)                        \notag                           \\
                                                        & \leq C \left( R^\alpha \|h\|_{0, \alpha; Q_{R/4, 6R}^{\tau_0, \tau_0 + 1}} + \|u\|_{0; Q_{R/4, 6R}^{\tau_0, \tau_0 + 1}}\right) \notag                           \\
                                                        & \leq CR^{d}e^{-\delta \tau_0} \left( \|h\|_{0, \alpha}^{[d-\alpha], \delta} + \|u\|_0^{[d], \delta}\right). \label{eq:our-parabolic-Schauder-estimate-late-time}
        \end{align}

        We first prove part \ref{item:our-spatial-Schauder-estimate}.
        The spatial H\"older norm of $u$ is estimated by $v$ as
        \begin{align*}
            \sup_{\tau \in [\tau_0 + T, \tau_0 + 1]} \|u(\cdot, \tau)\|_{2, \alpha; B_{2R} \setminus B_R} & \leq C \sup_{t \in [-e^{-T}, -1/e]} \|v(\cdot, t)\|_{2, \alpha; B_{2R} \setminus B_{R/2}} \\
                                                                                                          & \leq C \|v\|_{2, \alpha; Q_{R/2, 2R}^{-e^{-T}, -1/e}}
        \end{align*}
        for $T \in [0, 1/2]$.
        Combining this with \eqref{eq:our-parabolic-Schauder-estimate-initial} and \eqref{eq:our-parabolic-Schauder-estimate-late-time} gives
        \begin{align*}
            \sup_{\tau \in [0, 1]} \|u(\cdot, \tau)\|_{2, \alpha; B_{2R} \setminus B_R} \leq C R^d \left( \|u(\cdot, 0)\|_{2, \alpha}^{[d]} + \|h\|_{0, \alpha}^{[d-\alpha], \delta} + \|u\|_0^{[d], \delta}\right),
        \end{align*}
        and for all $\tau_0 \in \RR_+$,
        \begin{align*}
            \sup_{\tau \in [\tau_0 + 1/2, \tau_0 + 1]} \|u(\cdot, \tau)\|_{2, \alpha; B_{2R} \setminus B_R} \leq C R^d e^{-\delta \tau_0} \left( \|h\|_{0, \alpha}^{[d-\alpha], \delta} + \|u\|_0^{[d], \delta}\right).
        \end{align*}
        Hence, we get the desired estimate of \ref{item:our-spatial-Schauder-estimate} with the interior estimate \eqref{eq:our-parabolic-Schauder-estimate-cpt}.

        Next, we prove part \ref{item:our-parabolic-Schauder-estimate}.
        The claim gives
        \begin{align*}
            \|u\|_{2, \alpha; (B_{2R} \setminus B_R) \times [\tau_0 + T, \tau_0 + 1]} & \leq CR^{1 + \alpha} \|v\|_{2, \alpha; (B_{2R} \setminus B_{R/2}) \times [-e^{-T}, -1/e]}
        \end{align*}
        for $T \in [0, 1/2]$.
        In the same manner as the proof of part \ref{item:our-spatial-Schauder-estimate}, we get the desired estimate of \ref{item:our-parabolic-Schauder-estimate}.
    \end{proof}
\end{lemma}

Recall the following definition of the norm:
\begin{align*}
    \|u\|_* \defeq \|u\|^{[\gamma + \alpha], \delta}_{2, \alpha} + \|u\|_0^{[-1 + \gamma], \delta} + \|Du\|_0^{[-1 + \gamma], \delta},
\end{align*}
where $0 < \alpha < \gamma < 1/100$ are constants and $\delta>0$ is a small constant to be determined later.

\begin{corollary} \label{cor:our-norm-estimate-from-h-u0-u-norm}
    Suppose $u, h$ satisfy \eqref{eq:inhomogeneous-linear-PDE}.
    Let $\delta>0$.
    Then, we have
    \begin{align} \label{eq:our-norm-estimate-from-h-u0-u-norm}
        \|u\|_* \leq C\left(\|u(\cdot, 0)\|_{2, \alpha}^{[-1 + \gamma]} + \|h\|_{0, \alpha}^{[-1 + \gamma - \alpha], \delta} + \|u\|_0^{[-1 + \gamma], \delta}\right).
    \end{align}
    \begin{proof}
        For the sake of simplicity, we denote the right hand side of \eqref{eq:our-norm-estimate-from-h-u0-u-norm} by $H$.
        Lemma \ref{lem:our-holder-estimate} with $d = -1 + \gamma$ gives
        \begin{align*}
            \sup_{\tau \in [0, \infty)} \left[e^{\delta \tau} \|u(\cdot, \tau)\|_{2, \alpha}^{[-1 + \gamma]}\right] \leq CH
        \end{align*}
        and
        \begin{align*}
            \|u\|_{2, \alpha}^{[\gamma + \alpha], \delta} \leq CH.
        \end{align*}
        Then, all the terms of $\|u\|_*$ are directly bounded by $CH$.
    \end{proof}
\end{corollary}

\begin{lemma} \label{lem:weighted-decay-holder-multiplication-inequality}
    Let $d_1, d_2, \delta_1, \delta_2 \in \RR$ be arbitrary.
    Then, we have
    \begin{align*}
        \|fg\|^{[d_1+d_2], \delta_1 + \delta_2}_{0, \alpha} \leq C (\|f\|_{0, \alpha}^{[d_1], \delta_1}\|g\|_{0}^{[d_2], \delta_2} + \|f\|_{0}^{[d_1], \delta_1}\|g\|_{0, \alpha}^{[d_2], \delta_2}).
    \end{align*}
    for $f, g$ are functions on $\RR^n \times [0, \infty)$.
    \begin{proof}
        The result directly comes from the following inequality:
        \begin{align*}
            \|fg\|_{0, \alpha; Q(y, \tau)} \leq C(\|f\|_{0, \alpha; Q(y, \tau)}\|g\|_{0; Q(y, \tau)} + \|f\|_{0; Q(y, \tau)}\|g\|_{0, \alpha; Q(y, \tau)}). \quad \qedhere
        \end{align*}
    \end{proof}
\end{lemma}
The nonlinear term is estimated by the following lemma.
\begin{lemma} \label{lem:nonlinear-estimate}
    Let $0 < \alpha < \gamma < 1/100$ and $\delta > 0$ be fixed.
    There exists a constant $\eta = \eta(d)$ such that for arbitrary $u, \bar{u} \in \Omega^1(\mf{so}(n))$ with $\|u\|_*,  \leq \eta$ and $\|\bar{u}\|_* \leq \eta$:
    \begin{align*}
        \|\mc N(u)\|^{[-1-\gamma], 2\delta}_{0, \alpha} \leq C\|u\|_{*}^2,
    \end{align*}
    and
    \begin{align*}
        \|\mc N(\bar{u}) - \mc N(u)\|_{0, \alpha}^{[-1 -\gamma], 2\delta} \leq C (\|\bar{u}\|_* + \|u\|_*) \|\bar{u} - u\|_*.
    \end{align*}
    \begin{proof}
        We bound the terms $\|Du\|_{0, \alpha}^{[-2\gamma], \delta}$ and $\|u\|_{0, \alpha}^{[-2\gamma], \delta}$ in terms of $\|u\|_*$ using interpolation inequalities.
        Let $(y, \tau) \in \RR^n \times [0, \infty)$ be arbitrary.
        Denote $\tilde{r} = \tilde{r}(y)$ for convenience.
        By applying the interpolation inequality, we obtain
        \begin{align*}
            [Du]_{\alpha; Q(y, \tau)} & \leq C(\epsilon[u]_{2, \alpha; Q(y, \tau)} + \epsilon^{-(1 + \alpha)}\|u\|_{0;Q(y, \tau)})                              \\
                                      & \leq C\|u\|_*e^{-\delta \tau}(\epsilon \tilde{r}^{\gamma + \alpha} + \epsilon^{-(1 + \alpha)}\tilde{r}^{(-1 + \gamma)})
        \end{align*}
        for all $\epsilon > 0$.
        We choose $\epsilon = \tilde{r}^{-3\gamma - \alpha}$.
        Condition $0 < \alpha < \gamma < 1/100$ implies
        \begin{align*}
            [Du]_{\alpha; Q(y, \tau)} & \leq C\|u\|_* e^{-\delta \tau} (\tilde{r}^{-2\gamma} + \tilde{r}^{(1 + \alpha)(3\gamma + \alpha) - 1 + \gamma}) \\
                                      & \leq C\|u\|_*e^{-\delta \tau} \tilde{r}^{-2\gamma}.
        \end{align*}
        Hence, we get
        \begin{align*}
            \|Du\|_{0, \alpha}^{[-2\gamma], \delta} \leq C\|u\|_*.
        \end{align*}
        The interpolation inequality gives
        \begin{align*}
            [u]_{\alpha; Q(y, \tau)} & \leq C(\epsilon[u]_{2, \alpha; Q(y, \tau)} + \epsilon^{-\alpha/2}\|u\|_{0; Q(y, \tau)})                              \\
                                     & \leq C\|u\|_*e^{-\delta \tau}(\epsilon \tilde{r}^{\gamma + \alpha} + \epsilon^{-\alpha/2}\tilde{r}^{(-1 + \gamma)}).
        \end{align*}
        If we choose $\epsilon = \tilde{r}^{-3\gamma - \alpha}$, we get
        \begin{align*}
            [u]_{\alpha ;Q(y, \tau)} & \leq C\|u\|_*e^{-\delta\tau}( \tilde{r}^{-2\gamma} + \tilde{r}^{(\alpha/2)(3\gamma + \alpha) - 1 + \gamma}) \\
                                     & \leq C\|u\|_*e^{-\delta\tau}\tilde{r}^{-2\gamma}.
        \end{align*}
        Finally, we obtain
        \begin{align*}
            \|u\|_{0, \alpha}^{[-2\gamma], \delta} \leq C\|u\|_*.
        \end{align*}
        Recall the nonlinear term $\mc N(u)$ is given by
        \begin{align*}
            \mc N(u)_j & \defeq 2[u_i, \partial_i u_j + [W_i, u_j]] - [u_i, \partial_j u_i + [W_j, u_i]] - [u_i, [u_j, u_i]].
        \end{align*}
        Then, the result comes straightforward from inequalities
        \begin{align*}
             & \|W\|_{0, \alpha}^{[-1], 0} \leq C,                                                                          \\
             & \|u\|_{0, \alpha}^{[-2\gamma], \delta} \leq C\|u\|_*, \quad \|u\|_{0}^{[-1 + \gamma], \delta} \leq \|u\|_*   \\
             & \|Du\|_{0, \alpha}^{[-2\gamma], \delta} \leq C\|u\|_*, \quad \|Du\|_{0}^{[-1 + \gamma], \delta} \leq \|u\|_*
        \end{align*}
        and Lemma \ref{lem:weighted-decay-holder-multiplication-inequality}.
    \end{proof}
\end{lemma}

\subsection{Eigenfunction estimates} \label{subsec:eigenfunction-estimates}
Analogous to the Section \ref{subsec:pointwise-estimates}, we will establish weighted H\"older estimates for eigenfunctions of the linear operator $-L$.
\begin{lemma} \label{lem:supnorm-estimate-eigenstate}
    Let $(\xi, \lambda)$ be an eigenpair of $-L$.
    Then, $\xi(y)$ has a growth at most $r^{2\lambda - 1}$, or equivalently
    \begin{align*}
        \xi(y) < C\tilde{r}^{2\lambda - 1}
    \end{align*}
    for some $C = C(\xi, \lambda)$.
    \begin{proof}
        Let $R>0$ be a sufficiently large number to be determined later.
        With the similar calculation as in \eqref{eq:dt-L-estimate}, we have
        \begin{align*}
            -L'\|\xi\|_F \leq 0 \text{ on } \RR^n \setminus B_R
        \end{align*}
        with a linear operator
        \begin{align*}
            L' = \Delta - \frac{1}{2}y \cdot \nabla - \frac{1}{2} + \lambda + \frac{C_0}{r^2}
        \end{align*}
        for some constant $C_0(n)$.
        Take a function $\psi = \alpha (r^{2\lambda - 1} - r^{2\lambda - 2})$ on a domain $\RR^n \setminus B_R$ with
        \begin{align*}
            \alpha = 2R^{-(2\lambda - 1)}\sup_{\partial B_R}\|\xi\|_F.
        \end{align*}
        Define $f = \xi - \psi$ and $f_+ = \max\{f, 0\}$.
        For sufficiently large $R$ depending only on $\lambda$, we have
        \begin{align*}
            f & = \xi - \alpha (r^{2\lambda - 1} - r^{2\lambda - 2}) \leq \xi - \tfrac{1}{2}\alpha r^{2\lambda - 1} \\
              & = \xi - \sup_{\partial B_R}\|\xi\|_F \leq 0 \text{ on } \partial B_R.
        \end{align*}
        Moreover, for sufficiently large $R$ depending only on $n$ and $\lambda$, we get
        \begin{align}
            -L'f & \leq L'\psi \leq \alpha(C_0 + (2\lambda -1)(n + 2\lambda - 3))r^{2\lambda-3}                                       \notag  \\
                 & - \alpha(\tfrac{1}{2}r^{2\lambda - 2} + (C_0 + (2\lambda-2)(n + 2\lambda - 4))r^{2\lambda - 4})                     \notag \\
                 & \leq -\tfrac{1}{4}\alpha r^{2\lambda - 2} \leq 0 \text{ on } \RR^n \setminus B_R, \label{eq:eigenstate-L'-inequality}
        \end{align}
        Similar to \eqref{eq:barrier-tau-derivative-inequality-integrated}, multiply \eqref{eq:eigenstate-L'-inequality} by $f_+$, integrate over $\RR^n \setminus B_R$ gives
        \begin{align*}
            0 \leq \int_{\RR^n \setminus B_R} -|\nabla f_+|^2 \boldsymbol{\rho} dy^n + (\lambda + \frac{C_0}{R^2} - \frac{1}{2})\|f_+\|_{\boldsymbol{\rho}; \RR^n \setminus B_R}^2.
        \end{align*}
        Note that Ecker's Sobolev inequality \cite{Ecker} implies
        \begin{align*}
            \|yf_+\|_{\boldsymbol{\rho}} \leq 4n\|f_+\|_{\boldsymbol{\rho}, 1},
        \end{align*}
        and we thus estimate
        \begin{align*}
            0 & \leq \int_{\RR^n \setminus B_R} -|\nabla f_+|^2 \boldsymbol{\rho} dy^n + (\lambda + \frac{C_0}{R^2} - \frac{1}{2})\|f_+\|_{\boldsymbol{\rho}; \RR^n \setminus B_R}^2       \\
              & \leq \int_{\RR^n \setminus B_R} -\frac{1}{4n}|y f_+|^2\boldsymbol{\rho} dy^n + (\frac{1}{2} + \lambda + \frac{C_0}{R^2})\|f_+\|_{\boldsymbol{\rho}; \RR^n \setminus B_R}^2 \\
              & \leq (-\frac{1}{4n}R^2 + \frac{1}{2} + \lambda + \frac{C_0}{R^2})\|f_+\|_{\boldsymbol{\rho}; \RR^n \setminus B_R}^2.
        \end{align*}
        We get $f_+ \equiv 0$ after taking sufficiently large $R$.
    \end{proof}
\end{lemma}
The H\"older estimate for eigenfunctions follows from the rescaling argument.
\begin{lemma} \label{lem:eigenfunction-schauder-estimate}
    Let $(\lambda, \xi)$ be an eigenpair of $-L$.
    Then, $\xi \in C^{(2\lambda-1)}_{2, \alpha}(\RR^n; \Omega^1(\mf{so}(n)))$.
    \begin{proof}
        It suffices to estimate the norm of $\xi$ on the complement of the compact ball.
        Define a rescaled function $v(x, t)$ as
        \begin{align*}
            v(x, t) \defeq (-t)^{\lambda - 1/2} \xi(\frac{x}{\sqrt{-t}})
        \end{align*}
        for $t \in [-1, 0)$.
        The eigenfunction equation $-L\xi = \lambda \xi$ becomes
        \begin{align*}
            \frac{\partial v}{\partial t} = \Delta_x v + 2[\hat{W}_i, \frac{\partial}{\partial x^i}v_j] + [\hat{W}_i, [\hat{W}_i, v_j]] + 2[v_i, \hat{F_W}_{ij}]
        \end{align*}
        where we used the rescaled variables $\hat{W}, \hat{F_W}$ defined in \eqref{eq:rescaled-W-F}.
        Since
        \begin{align*}
            W_j(y) = p_j(y) / (a|y|^2 + b)  \quad \text{ and } \quad {F_{W}}_{ij}(y) = q_{ij}(y)/ (a|y|^2 + b)^2
        \end{align*}
        for polynomials $p_j, q_{ij}$ with degree $1$ and $2$ respectively, we obtain the H\"older estimate of coefficients
        \begin{align*}
            \sup_{t \in [0, 1)}\left[\|\hat{W}(\cdot, t)\|_{0, \alpha; B_4 \setminus B_1} + \|\hat{F_W}(\cdot, t)\|_{0, \alpha; B_4 \setminus B_1}\right] \leq C.
        \end{align*}
        Applying Knerr's Schauder estimates(see Theorem \ref{thm:knerr-Schauder-estimate}) on $(B_4 \setminus B_1) \times [-1, 0)$ gives
        \begin{align*}
            \|\xi\|^{(2\lambda - 1)}_{2, \alpha; \RR^n \setminus B_{2\sqrt{2}}} & \leq C\sup_{t \in [-1/2, 0)}\left[\|\xi\|_{2, \alpha; B_{3/\sqrt{-t}} \setminus B_{2/\sqrt{-t}}}^{(2\lambda - 1)}\right] = C\sup_{t \in [-1/2, 0)}\|v(\cdot, t)\|_{2, \alpha; B_{3} \setminus B_{2}}^{(2\lambda - 1)} \\
                                                                                & \leq C\sup_{t \in [-1/2, 0)}\|v(\cdot, t)\|_{2, \alpha; B_{3} \setminus B_{2}} \leq C \|v\|_{0; Q_{1, 4}^{-1, 0}} \leq C \|\xi\|_{0}^{(2\lambda - 1)}. \qedhere
        \end{align*}
    \end{proof}
\end{lemma}
\begin{corollary} \label{cor:positive-projection-preserves-para-holderness}
    Let $d \geq -1$.
    Then, the projection $\Pi_{>0}$ maps $C^{(d)}_{2, \alpha}(\RR^n; \Omega^1(\mf{so}(n)))$ into itself.
    Moreover, $\Pi_{>0}$ is a Lipschitz map. \ie For $u \in C^{(d)}_{2, \alpha}(\RR^n; \Omega^1(\mf{so}(n)))$,
    \begin{align*}
        \|\Pi_{>0} u\|_{2, \alpha}^{(d)} \leq C(d)\|u\|_{2, \alpha}^{(d)}.
    \end{align*}
    \begin{proof}
        By Lemma \ref{lem:eigenfunction-schauder-estimate},
        \begin{align*}
            \|\Pi_{>0} u\|_{2, \alpha}^{(d)} & \leq \|u\|_{2, \alpha}^{(d)} + \sum_{i=1}^{I} \|\Pi_{=i} u\|_{2, \alpha}^{(d)}                                                    \\
                                             & = \|u\|_{2, \alpha}^{(d)} + \sum_{i=1}^{I} \|u\|_{\boldsymbol{\rho}}\|\xi_i\|_{2, \alpha}^{(-1)} \leq C(d)\|u\|_{2, \alpha}^{(d)}
        \end{align*}
        gives the result.
    \end{proof}
\end{corollary}

\begin{lemma} \label{lem:linear-parabolic-PDE-*norm-estimate}
    Let $0 < \delta < \min\{\lambda_{I+1}, \gamma/2\}$ and $-1 \leq d \leq -1 + \gamma$.
    Let $\tilde{v}_0 \in \Omega^1(\mf{so}(n))$ with $\|\tilde{v}_0\|_{2, \alpha}^{(d)}<\infty$.
    Then, there exists a solution $v$ to a homogeneous linear parabolic PDE
    \begin{align*}
        \frac{\partial v}{\partial \tau} = Lu
    \end{align*}
    with an initial condition $v(\cdot, 0) = \Pi_{>0} \tilde{v}_0$.
    Furthermore, we have
    \begin{align*}
        \|v\|_* \leq C(d)\|\tilde{v}_0\|_{2, \alpha}^{(d)}.
    \end{align*}
    \begin{proof}
        Let $v_0 = \Pi_{>0} \tilde{v}_0$ and $\bar{v} = v_0 e^{-2\lambda_{I+1}\tau}$.
        By direct calculation, we have
        \begin{align} \label{eq:bar-v-holder-estimate}
            \|(L - \frac{\partial}{\partial \tau})\bar{v}\|^{[d], 2\lambda_{I+1}}_{0, \alpha} \leq C\|v_0\|_{2, \alpha}^{(d)},
        \end{align}
        and this implies
        \begin{align*}
            \int_0^\infty |e^{\delta_1 \tau}\|((\frac{\partial}{\partial \tau} - L)\bar{v})(\cdot, \tau)\|_\rho|^2 d\tau & < C (\|v_0\|^{(d)}_{2, \alpha})^2
        \end{align*}
        for some $\delta_1 \in (\delta, \lambda_{I+1})$.
        Lemma \ref{lem:existence-unique-sol-parabolic-PDE} implies that there exists a unique solution $v$ to the equations
        \begin{align*}
             & \frac{\partial (v - \bar{v})}{\partial \tau} = L(v - \bar{v}) - (\frac{\partial}{\partial \tau} - L)\bar{v} \quad\text{ or }\quad \frac{\partial v}{\partial \tau} = L v, \\
             & \Pi_{>0}((v - \bar{v})(\cdot, 0)) = 0,
        \end{align*}
        satisfying
        \begin{align}
            \int_0^\infty |e^{\delta \sigma} (\|v(\cdot, \sigma)\|_{\boldsymbol{\rho}} + \|\frac{\partial v}{\partial \sigma}(\cdot, \sigma)\|_{\boldsymbol{\rho}})|^2 d\sigma < \infty, \notag \\
            e^{\delta\tau}\|v(\cdot, \tau)\|_{\boldsymbol{\rho}} < C \|v_0\|_{2, \alpha}^{(d)}. \label{eq:bar-v-rho-estimate}
        \end{align}
        for all $\tau \in \RR_+$.
        Since inequality \eqref{eq:bar-v-rho-estimate} gives exponential decay of $\|v(\cdot, \tau)\|_{\boldsymbol{\rho}}$, we have $\Pi_{\leq 0} v = 0$.
        Hence,
        \begin{align*}
            v(y, 0) = \bar{v}(y, 0) = v_0(y, 0)
        \end{align*}
        for all $y \in \RR^n$.
        Applying Lemma \ref{lem:weighted-holder-estimate-on-cpt} to $(u, h) = (v - \bar{v}, -(\frac{\partial}{\partial \tau} - L)\bar{v})$ with \eqref{eq:bar-v-holder-estimate} gives an interior estimate
        \begin{align*}
            \|v\|_{2, \alpha; Q_R}^{\delta} & \leq \|v - \bar{v}\|_{2, \alpha; Q_R}^{\delta} + \|\bar{v}\|_{2, \alpha; Q_R}^{\delta}                              \\
                                            & \leq C(R)(\|(\frac{\partial}{\partial \tau} - L)\bar{v}\|^{2\lambda_{I+1}}_{0, \alpha} + \|v_0\|_{2, \alpha}^{(d)}) \\
                                            & \leq C(R) \|v_0\|_{2, \alpha}^{(d)}.
        \end{align*}
        for sufficiently large $R$.
        Plugging Lemma \ref{lem:supnorm-estimate} \ref{item:exp-decay-supnorm-estimate} with $(u, h, d, \delta) = (v, 0, -1 + \gamma, \delta)$, we obtain
        \begin{align*}
            \|v\|_0^{[-1 + \gamma], \delta} \leq C \|v_0\|_{2, \alpha}^{(d)}.
        \end{align*}
        Hence, Corollary \ref{cor:our-norm-estimate-from-h-u0-u-norm} implies applied to $(u, h) = (v, 0)$ and Corollary \ref{cor:positive-projection-preserves-para-holderness} implies
        \begin{align*}
            \|v\|_* & \leq C (\|v(\cdot, 0)\|_{2, \alpha}^{[-1 + \gamma]} + \|v\|_0^{[-1 + \gamma], \delta}) \\
                    & \leq C\|v_0\|_{2, \alpha}^{(d)} \leq C(d) \|\tilde{v}_0\|_{2, \alpha}^{(d)}. \qedhere
        \end{align*}
    \end{proof}
\end{lemma}
\section{Main theorem} \label{sec:main-theorem}
We fix $\delta \in (0, \min\{\gamma/2, \lambda_{I+1}\})$, and $-1 \leq d \leq -1 + \gamma$.
Define a map
\begin{align*}
    \iota_+: C_{2, \alpha}^{(d)}(\Omega^1(\mf{so}(n))) & \to \Omega^1(\mf{so}(n)) \times \RR_+ \\
    \tilde{v}_0                                        & \mapsto v
\end{align*}
where $v$ is the solution we obtained in the Lemma \ref{lem:linear-parabolic-PDE-*norm-estimate}.
\begin{theorem} \label{thm:rescaled-de-turck-main-thm}
    There exists $\mu_0 = \mu_0(\delta, \alpha)$ such that, for every $\mu > \mu_0$, there exists a corresponding $\epsilon = \epsilon(\alpha, \mu) > 0$ with the following property:
    For any $\tilde{u}_0 \in \Omega^1(\mf{so}(n))$ with $\|\tilde{u}_0\|_{2, \alpha}^{(d)} < \epsilon$, there exists a unique smooth solution $\mc T(\tilde{u}_0) \in \Omega^1(\mf{so}(n)) \times \RR_+$ to the rescaled Yang--Mills de-Turck flow equation \eqref{eq:transformed-YM-flow-3} with a priori decay
    \begin{align*}
        \|\mc T(\tilde{u}_0) - \iota_{+}(\tilde{u}_0)\|^{[-1 + \gamma], \delta}_{0} \leq \mu (\|\tilde{u}_0\|_{2, \alpha}^{(d)})^2
    \end{align*}
    and the initial condition $\Pi_{>0}\mc T(\tilde{u}_0)(\cdot, 0) = \Pi_+(\tilde{u}_0)$.
    Moreover, if $d = -1$, we have a decay of the solution in the spatial direction:
    \begin{align*}
        \|\mc T(\tilde{u}_0) - \iota_+(\tilde{u}_0)\|_{0}^{[-1]} \leq C(\|\tilde{u}_0\|_{2, \alpha}^{(-1)})^2.
    \end{align*}
    \begin{proof}
        Consider the Banach space
        \begin{align*}
            \mc C[\tilde{u}_0] \defeq \Set{u \in \Omega^1(\mf{so}(n)) \times \RR_+ \given \|u\|_{*} < \infty, \Pi_{>0}(u - \tilde{u}_0) = 0}.
        \end{align*}
        From Lemma \ref{lem:linear-parabolic-PDE-*norm-estimate}, we have $\|\iota_+(\tilde{u}_0)\|_* \leq C\|\tilde{u}_0\|_{2, \alpha}^{(d)}$.
        Let $\eta > 0$ be the constant in Lemma \ref{lem:nonlinear-estimate}.
        For $u \in \mc C[\tilde{u}_0]$ with $\|u\|_{*} < \eta$, the nonlinear estimate in Lemma \ref{lem:nonlinear-estimate} implies
        \begin{align} \label{eq:nonlinear-estimate-thm}
            \|\mc N(u)\|_{0, \alpha}^{[-1 - \gamma], 2\delta} \leq C(\|u\|_*)^2.
        \end{align}
        Hence, there is a unique solution $\mc T(u; \tilde{u}_0)$ of
        \begin{align*}
            (\frac{\partial}{\partial \tau} - L)(\mc T(u; \tilde{u}_0) - \iota_+(\tilde{u}_0)) = \mc N(u), \quad \Pi_{>0}(\mc T(u; \tilde{u}_0) - \iota_+(\tilde{u}_0)) = 0
        \end{align*}
        by Lemma \ref{lem:existence-unique-sol-parabolic-PDE}.
        From \eqref{eq:u-j-for-j-leq-I}, we have
        \begin{align*}
            \langle (\mc T(u; \tilde{u}_0) - \iota_+(\tilde{u}_0))(\cdot, 0), \xi_j\rangle_{\boldsymbol{\rho}} = \int_0^\infty e^{\lambda_j \sigma} \|\mc N(u)\|_{\boldsymbol{\rho}} d\sigma \leq C (\|u\|_*)^2
        \end{align*}
        for $1 \leq j \leq I$.
        Since $\|\cdot\|_{2, \alpha}^{(-1)}$-norm of eigenfunctions associated with non-positive eigenvalues are bounded in Lemma \ref{lem:eigenfunction-schauder-estimate}, we obtain an estimate of the initial boundary:
        \begin{align} \label{eq:tu-uj-t-eq-0-estimate}
            \|(\mc T(u; \tilde{u}_0) - \iota_+(\tilde{u}_0))(\cdot, 0)\|_{2, \alpha}^{(-1)} \leq C (\|u\|_*)^2.
        \end{align}
        Applying Lemma \ref{lem:weighted-holder-estimate-on-cpt} with $(u, h) = (\mc T(u; \tilde{u}_0) - \iota_+(\tilde{u}_0), \mc N(u))$ gives an parabolic H\"older estimate on the cylinder $B_R \times \RR_+$ for some $R>1$:
        \begin{align}
            \|\mc T(u; \tilde{u}_0) - \iota_+(\tilde{u}_0)\|_{2, \alpha; B_R \times \RR_+}^{\delta} & \leq C (\|\mc N(u)\|_{0, \alpha}^{2\delta} + \|(\mc T(u; \tilde{u}_0) - \iota_+(\tilde{u}_0))(\cdot, 0)\|_{2, \alpha; B_{R+1}}) \notag \\
                                                                                                    & \leq C (\|u\|_*)^2.\label{eq:tu-u-cpt-supnorm-estimate}
        \end{align}
        where we used \eqref{eq:nonlinear-estimate-thm} and \eqref{eq:tu-uj-t-eq-0-estimate}.
        The global supnorm estimate is achieved from Lemma \ref{lem:supnorm-estimate} \ref{item:exp-decay-supnorm-estimate} based on \eqref{eq:nonlinear-estimate-thm}, \eqref{eq:tu-u-cpt-supnorm-estimate} and \eqref{eq:tu-uj-t-eq-0-estimate} as
        \begin{align}
            \|(\mc T(u; \tilde{u}_0) - \iota_+(\tilde{u}_0))\|_{0}^{[-1+\gamma], \delta} & \leq C(\|\mc N(u)\|_{0}^{[-1 + \gamma], \delta} + \|\mc T(u; \tilde{u}_0) - \iota_+(\tilde{u}_0)\|_{0; Q_R}^{[-1 + \gamma], \delta} \notag \\
                                                                                         & \quadn{5}+ \|(\mc T(u; \tilde{u}_0) - \iota_+(\tilde{u}_0))(\cdot, 0)\|^{[-1 + \gamma]}_0) \notag                                          \\
                                                                                         & \leq C(\|u\|_*)^2.\label{eq:tu-u-supnorm-estimate}
        \end{align}
        for $\tau \in \RR_+$.
        Finally, Corollary \ref{cor:our-norm-estimate-from-h-u0-u-norm} and inequalities \eqref{eq:tu-uj-t-eq-0-estimate}, \eqref{eq:nonlinear-estimate-thm}, and \eqref{eq:tu-u-supnorm-estimate} imply for $\tau \in \RR_+$,
        \begin{align}
            \|\mc T(u; \tilde{u}_0) - \iota_+(\tilde{u}_0)\|_{*} & \leq C (\|(\mc T(u; \tilde{u}_0) - \iota_+(\tilde{u}_0))(\cdot, 0)\|_{2, \alpha}^{[-1+\gamma]}+ \|\mc N(u)\|_{0, \alpha}^{[-1 + \gamma - \alpha], \delta}                                                                                             \notag \\
                                                                 & \quadn{8}+ \|\mc T(u; \tilde{u}_0) - \iota_+(\tilde{u}_0)\|_0^{[-1+\gamma], \delta}) \notag                                                                                                                                                                  \\
                                                                 & \leq C (\|u\|_*)^2.\label{eq:tu-u-star-estimate}
        \end{align}
        Thus, $\mc T$ maps small ball in $\mc C[\tilde{u}_0]$ to itself.
        Likewise, for $\bar{u} \in \mc C[\tilde{u}_0], \|\bar{u}\|_* \leq \eta$, we have
        \begin{align*}
             & (\frac{\partial}{\partial \tau} - L)(\mc T(\bar{u}; \tilde{u}_0) - \mc T(u; \tilde{u}_0)) = \mc N(\bar{u}) - \mc N(u), \\
             & \quadn{4}\Pi_{>0}(\mc T(\bar{u}; \tilde{u}_0) - \mc T(u; \tilde{u}_0))(\cdot, 0) = 0.
        \end{align*}
        With the same discussion above applies with $\bar{u} - u$ in place of $u - \iota_+(\tilde{u}_0)$, we have
        \begin{align} \label{eq:tu-baru-star-estimate}
            \|\mc T(\bar{u}; \tilde{u}_0) - \mc T(u; \tilde{u}_0)\|_{*} \leq C (\|\bar u\|_{*} + \|u\|_{*}) \|\bar{u} - u\|_{*}.
        \end{align}
        Consider the subset $X \defeq \Set{u \in \mc C[\tilde{u}_0] \given \|u - \iota_+(\tilde{u}_0)\|_* \leq \mu (\|\tilde{u}_0\|_{2, \alpha}^{(d)})^2}$.
        There exists $\mu_0 = \mu_0(\delta, \alpha)$ such that for every $\mu > \mu_0$, there exists a corresponding $\epsilon = \epsilon(\alpha, \mu) > 0$ with the following property:
        For any $\tilde{u}_0 \in C_{2, \alpha}^{(d)}(\RR^n; \Omega^1(\mf{so}(n)))$ with $\|\tilde{u}_0\|_{2, \alpha}^{(d)} < \epsilon$ and $u \in X$ imply $\mc T(u; \tilde{u}_0) \in X$ by triangular inequality, \eqref{eq:tu-u-star-estimate} and Lemma \ref{lem:linear-parabolic-PDE-*norm-estimate}.

        By \eqref{eq:tu-baru-star-estimate} $\mc T$ is a contraction map on $X$.
        By the completeness of $X$, there is a fixed point $\mc T(\tilde{u}_0) \in X$ satisfying the rescaled Yang--Mills de-Turck flow equation \eqref{eq:transformed-YM-flow-3}.
        We get the smoothness of $\mc T(\tilde{u}_0)$ from the standard parabolic Schauder bootstrapping.

        If $d = -1$, Lemma \ref{lem:supnorm-estimate} \ref{item:-1-supnorm-estimate} gives the optimal decay of $\mc T(\tilde{u}_0) - \iota_+(\tilde{u}_0)$ in the spatial direction as
        \begin{align*}
            \|\mc T(\tilde{u}_0) - \iota_+(\tilde{u}_0)\|_0^{[-1]} & \leq C(n, \gamma)(\|\mc T(\tilde{u}_0) - \iota_+(\tilde{u}_0)\|^{[-1]}_{0; Q_R} + \|\mc N(\mc T(\tilde{u}_0))\|_{0}^{[-1-\gamma], 0} \\
                                                                   & \quadn{8} + \|(\mc T(\tilde{u}_0) - \iota_+(\tilde{u}_0))(\cdot, 0)\|_0^{[-1]})                                                      \\
                                                                   & \leq C(n, \gamma)(\|\tilde{u}_0\|_{2, \alpha}^{(-1)})^2. \qedhere
        \end{align*}
    \end{proof}
\end{theorem}
We get the main theorem as a corollary.
\begin{theorem} \label{thm:main-theorem}
    There exist constants $\epsilon, M, \delta > 0$ such that the following holds.
    For any $\tilde{A}_0 \in C_{2, \alpha}^{(d)}(\Omega^1(\mf{so}(n)))$ with $\|\tilde{A}_0\|_{2, \alpha}^{(d)} < \epsilon$,
    there exists a smooth solution $A(x, t)$ to the Yang--Mills flow equation \eqref{eq:YM-flow}, which blowups at $t = 1$, with the initial condition:
    \begin{align*}
        \Pi_{>0}(A(\cdot, 0) - W - \tilde{A}_0) = 0.
    \end{align*}
    Furthermore, there exists a gauge transformation $S(t)$, such that the gauge transformation of the solution $S^*A(x, t)$ converges to $\bf{W}$ in the following sense:
    \begin{align*}
        \frac{\|(\bf{W}(\cdot, t) - S(t)^*A)(\cdot, t)\|_{L^\infty}}{\|\bf{W}(\cdot, t)\|_{L^\infty}} < M \|\tilde{A}_0\|^{(d)}_{2, \alpha} (1-t)^{\delta}
    \end{align*}
    for all $t \in [0, \infty)$.
\end{theorem}
\begin{proof}
    Let $u(y, \tau)$ be the solution to the rescaled Yang--Mills de-Turck flow equation obtained in Theorem \ref{thm:rescaled-de-turck-main-thm}.
    Recall that the solution to the Yang--Mills flow equation is given by
    \begin{align*}
        A(x, t) = S\psi S^{-1}(x, t) + S d_x S^{-1}(x, t)
    \end{align*}
    where $S$ is the gauge transformation defined in \eqref{eq:gauge-transformation-ode} and
    \begin{align*}
        \psi(x, t) = \frac{1}{\sqrt{1 - t}} u(\frac{x}{\sqrt{1 - t}}, -\log(1-t)).
    \end{align*}
    Then, we get the estimate
    \begin{align*}
        \|\sqrt{1 - t}(S^*A - \bf{W})_i(\cdot, t)\|_0 \leq C \|u(\cdot, -\log(1-t))\|_0 \leq C \|\tilde{A}_0\|_{2, \alpha}^{(d)} (1-t)^{\delta}. \quad \qedhere
    \end{align*}
\end{proof}
We obtain Corollary \ref{cor:non-equivariant-solution} using the above theorems.
\begin{proof}[Proof of Corollary \ref{cor:non-equivariant-solution}]
    During the construction of a non-$SO(n)$-equivariant solution, $R>0$ will be chosen larger, and $\epsilon_0>0$ will be chosen smaller.
    Fix $-1 < d \leq -1 + \gamma$, $\epsilon, \mu$ in Theorem \ref{thm:main-theorem}.
    Let $\chi \in C^\infty([0, \infty))$ be a cutoff function such that $\chi \equiv 1$ on $[0, 1]$ and $\chi \equiv 0$ on $[2, \infty)$.
    We define a smooth initial data $\tilde{u}_0$ by
    \begin{align*}
        \tilde{u}_0(x) = -W(x) \chi(|\frac{4(x - R)}{R}|).
    \end{align*}
    Note that $\|\tilde{u}_0\|_{2, \alpha}^{(d)}$ can be arbitrarily small if we take $R$ sufficiently large.
    For sufficiently large $R$, we have
    \begin{align} \label{eq:non-equivariant-initial-condition-estimate}
        \|\tilde{u}_0\|_{2, \alpha}^{(d)} < \epsilon_0
    \end{align}
    with $\epsilon_0 < \epsilon$.
    Let $u$ be the solution to the rescaled Yang--Mills de-Turck flow equation with the initial condition $\tilde{u}_0$ obtained in Theorem \ref{thm:rescaled-de-turck-main-thm}.
    We will show that $u$ cannot be gauge equivalent to an $SO(n)$-equivariant solution.
    We use the fact that the curvature component ${F_A}_{ij}(x, t)$ is invariant under the gauge transformation up to the adjoint action of $SO(n)$, and thus the Frobenius norm $\|{F_A}_{ij}\|_F$ is gauge invariant.
    This implies that, if $u$ is gauge equivalent to an $SO(n)$-equivariant solution, the following holds
    \begin{align*}
        \|{F_{W + u}}_{ij}(Re_1, 0)\|_F = \|{F_{W + u}}_{ij}(-Re_1, 0)\|_F
    \end{align*}
    where $e_1 = (1, 0, \cdots, 0) \in \RR^n$.
    for $1 \leq i, j \leq n$.

    By Lemma \ref{lem:eigenfunction-schauder-estimate}, we have
    \begin{align} \label{eq:non-equivariant-estimate-xi-j-holder}
        \|\xi_i\|_{2, \alpha}^{(-1)} < M_0
    \end{align}
    for $1 \leq i \leq I$ where $M_0>0$ is a constant.
    For sufficiently large $R$, we have
    \begin{align} \label{eq:non-equivariant-estimate-1}
        |\langle \tilde{u}_0(\cdot, 0), \xi_i\rangle_{\boldsymbol{\rho}}| \leq e^{-R^2/16} \pi (R/2)^2 \|W(x)\|_0 \|\xi_i(x)\|_0 < \frac{\epsilon_0}{M_0}
    \end{align}
    for $1 \leq i \leq I$.
    For sufficiently small $\epsilon_0$, estimates \eqref{eq:tu-uj-t-eq-0-estimate} and \eqref{eq:non-equivariant-initial-condition-estimate} imply
    \begin{align}
        |\langle u(\cdot, 0), \xi_i\rangle_{\boldsymbol{\rho}}| & \leq  C \|(u - \iota_+(\tilde{u}_0))(\cdot, 0)\|_{2, \alpha}^{(d)}\|\xi_i\|_0 \leq C(\|\tilde{u}_0\|_{2, \alpha}^{(d)})^2 M_0 \notag \\
                                                                & \leq  C \epsilon_0^2 M_0 < \frac{\epsilon_0}{M_0}\label{eq:non-equivariant-estimate-2}
    \end{align}
    for $1 \leq i \leq I$.
    Set $\bar{u} \in \Omega^1(\mf{so}(n))$ with $\bar{u}(x) \defeq u(x, 0) - \tilde{u}_0(x)$.
    Let $c_i$ be the coefficient of the expansion of $\bar{u}$ in terms of $\xi_i$:
    \begin{align*}
        \bar{u} = \sum_{i = 1}^{I} c_i \xi_i.
    \end{align*}
    Then $|c_i| < \frac{2\epsilon_0}{M_0}$ is obtained by \eqref{eq:non-equivariant-estimate-1} and \eqref{eq:non-equivariant-estimate-2}.
    Combining this with \eqref{eq:non-equivariant-estimate-xi-j-holder}, we get
    \begin{align}
        \|{F_{W + u}}_{ij}(Re_1, 0)\|_{F} & = \|{F_{\bar{u}}}_{ij}(Re_1, 0)\|_F = \|{F_{\bar{u}}}_{ij}(Re_1, 0)\|_F                                                                                        \notag \\
                                          & = \|(\partial_i \bar{u}_j - \partial_j \bar{u}_i + [\bar{u}_i, \bar{u}_j])(Re_1, 0)\|_F \leq \frac{C_0\epsilon_0}{R^2} \label{eq:non-equivariant-estimate-3}
    \end{align}
    for some constant $C_0>0$.
    On the other hand, we explicitly compute
    \begin{align*}
        \|{F_{W}}_{23}(-Re_1, 0)\|_F = \sqrt{2}\frac{(2a - 1)R^2 + 2b}{(aR^2 + b)^2},
    \end{align*}
    and thus
    \begin{align*}
        \|{F_{W}}_{23}(-Re_1, 0)\|_F > \frac{C_1}{R^2}
    \end{align*}
    for some $C_1 > 0$.
    From the similar computation as \eqref{eq:non-equivariant-estimate-3}, for some $C_2 > 0$,
    \begin{align*}
        \|{F_{W + u}}_{23}(-Re_1, 0)\|_F \geq \frac{C_1 - C_2\epsilon_0}{R^2}.
    \end{align*}
    Now, we choose $\epsilon_0$ small enough so that $\|{F_{W+u}}_{23}(Re_1, 0)\|_F < \|{F_{W+u}}_{23}(-Re_1, 0)\|_F$.
    Hence, $u$ cannot be gauge equivalent to an $SO(n)$-equivariant solution.
    Finally, Theorem \ref{thm:main-theorem} gives the desired result.
\end{proof}

\appendix
\renewcommand{\thesection}{\Alph{section}}
\renewcommand{\thetheorem}{\thesection.\arabic{theorem}}
\section{Parabolic Schauder estimates} \label{appendix:parabolic-Schauder-estimates}
Let $n, m \in \mathbb{N}$ and $B_R \subseteq \RR^n$ be the ball of radius $R$ centered at the origin.
We will work with an $\RR^m$-valued function $u$ and an elliptic operator $L$ in $Q_2 \defeq B_2 \times [-4, 0]$ of the form
\begin{align*}
    Lu := \Delta u + \sum_{i=1}^n b^iD_iu + cu,
\end{align*}
with coefficients $b^i, c : Q_2 \to \text{End}(\mathbb{R}^m)$ satisfying
\begin{align} \label{eq:elliptic-operator-coefficients-condition}
    \|b^i\|_{0,\alpha; Q_2} + \|c\|_{0,\alpha; Q_2} \leq \Lambda,
\end{align}
for some fixed constants $\Lambda > 0, \alpha \in (0, 1)$.
For the theorems that follows, we will assume that $u$ is a classical solution to
\begin{align*}
    \tfrac{\partial}{\partial t} u - Lu = h
\end{align*}
in $Q_2$.

We state the local parabolic Schauder estimates (see \cite[Appendix C]{Choi}), along with works by Schlag \cite{Schlag} and Simon \cite{Simon}.
\begin{theorem}[local parabolic Schauder estimate]\label{thm:parabolic-Schauder-estimate}
    The following estimate holds
    \begin{align*}
        \|u\|_{2, \alpha; B_1 \times [-1, 0]} \leq C \left(\|u\|_{0; Q_2} + \|h\|_{0, \alpha; Q_2}\right),
    \end{align*}
    for some constant $C = C(n, \lambda, \Lambda, \alpha) > 0$.
\end{theorem}
\begin{theorem}[local $L^2$-Schauder estimate] \label{thm:parabolic-Schauder-Lp-estimate}
    We have an $L^2$-Schauder estimate
    \begin{align*}
        \|u\|_{2, \alpha; B_1 \times [-1, 0]} \leq C  \left(\|u\|_{L^2(Q_2)} + \|h\|_{0, \alpha; Q_2}\right),
    \end{align*}
    for some constant $C = C(n, \lambda, \Lambda, \alpha) > 0$.
\end{theorem}
The following non-standard Schauder estimate holds due to Knerr:
\begin{theorem}\textup{\cite[Theorem 1]{Knerr}} \label{thm:knerr-Schauder-estimate}
    The following estimate holds
    \begin{align*}
        \sum_{j=0}^{2}\|\nabla_x^j u\|_{0, \alpha; B_1 \times [-1, 0]} + & \sup_{\tau \in [-1, 0]}\|\frac{\partial u}{\partial t}(\cdot, \tau)\|_{0, \alpha; B_1}             \\
                                                                         & \leq C \left(\|u\|_{0; Q_{2}} + \sup_{\tau \in [-4, 0]}\|h(\cdot, \tau)\|_{0, \alpha; B_2}\right),
    \end{align*}
    for some constant $C = C(n, \lambda, \Lambda, \theta) > 0$.
\end{theorem}
Also, we have the estimates for the initial boundary value problem:
\begin{theorem}[local parabolic Schauder estimate with initial data] \label{thm:parabolic-Schauder-initial-data-estimate}
    Let $T \in [1, \infty)$.
    Suppose $u, h$ is a classical solution to
    \begin{align*}
        \tfrac{\partial}{\partial t} u - Lu = h
    \end{align*}
    in $B_2 \times [0, T]$ and a condition \eqref{eq:elliptic-operator-coefficients-condition} holds for the domain $B_2 \times [0, T]$ instead of $Q_2$.
    Then, we have the following estimate
    \begin{align*}
        \|u\|_{2, \alpha; B_1 \times [0, T]} \leq C \left(\|u(\cdot, 0)\|_{2, \alpha; B_2} + \|u\|_{0; B_2 \times [0, T]} + \|h\|_{0, \alpha; B_2 \times [0, T]}\right),
    \end{align*}
    where $C = C(n, \lambda, \Lambda, \alpha) > 0$.
\end{theorem}
Following the proof of Corollary B.2 in \cite{CHODOSH} with Theorem \ref{thm:parabolic-Schauder-initial-data-estimate} gives the $L^2$-estimate for the initial boundary value problem:
\begin{corollary}[local $L^2$-Schauder estimate with initial data] \label{cor:parabolic-Schauder-Lp-initial-data-estimate}
    We have the following estimate
    \begin{align*}
        \|u\|_{2, \alpha; B_1 \times [-4, 0]} \leq C \left(\|u(\cdot, -4)\|_{2, \alpha; B_2} + \|u\|_{L^2(Q_2)} + \|h\|_{0, \alpha; Q_2}\right),
    \end{align*}
    where $C = C(n, \lambda, \Lambda, \alpha) > 0$.
\end{corollary}

\newcommand{\etalchar}[1]{$^{#1}$}


\begin{thebibliography}{CHHW22}
    \bibitem[ABDS22]{ABDS-Ricci-cylindrical-asymptotic}
    S.~Angenent, S.~Brendle, P.~Daskalopoulos, and N.~Sesum.
    \newblock Unique asymptotics of compact ancient solutions to three-dimensional ricci flow.
    \newblock {\em Comm. Pure Appl. Math.}, 75(5):1032--1073, 2022.

    \bibitem[ADS19]{ADS-asymptotics-MCF}
    S.~Angenent, P.~Daskalopoulos, and N.~Sesum.
    \newblock Unique asymptotics of ancient convex mean curvature flow solutions.
    \newblock {\em J. Differential Geom.}, 111(3):381--455, 2019.

    \bibitem[ADS20]{ADS-TWO-CONVEX}
    S.~Angenent, P.~Daskalopoulos, and N.~Sesum.
    \newblock Uniqueness of two-convex closed ancient solutions to the mean curvature flow.
    \newblock {\em Ann. of Math.}, 192(2):353--436, 2020.

    \bibitem[AV97]{Angenent-Degenerate-Neckpinches}
    S.~B. Angenent and J.~J.~L. Velazquez.
    \newblock Degenerate neckpinches in mean curvature flow.
    \newblock {\em J. Reine Angew. Math.}, 482:15--66, 1997.

    \bibitem[BC19]{Brendle-Choi-1}
    S.~Brendle and K.~Choi.
    \newblock Uniqueness of convex ancient solutions to mean curvature flow in $\mathbb{R}^3$.
    \newblock {\em Invent. Math.}, 217(1):35--76, 2019.

    \bibitem[BC21]{Brendle-Choi-2}
    S.~Brendle and K.~Choi.
    \newblock Uniqueness of convex ancient solutions to mean curvature flow in higher dimensions.
    \newblock {\em Geom. Topol.}, 25(5):2195--2234, 2021.

    \bibitem[BDNS23]{Brendle-Daskalopoulos-Naff-Sesum}
    S.~Brendle, P.~Daskalopoulos, K.~Naff, and N.~Sesum.
    \newblock Uniqueness of compact ancient solutions to the higher-dimensional ricci flow.
    \newblock {\em J. Reine Angew. Math.}, 2023(795):85--138, 2023.

    \bibitem[BDS21]{BDS-compact-ancient-Ricci-flow}
    S.~Brendle, P.~Daskalopoulos, and N.~Sesum.
    \newblock Uniqueness of compact ancient solutions to three-dimensional ricci flow.
    \newblock {\em Invent. Math.}, 226(2):579--651, 2021.

    \bibitem[BN23]{Brendle-Naff}
    S.~Brendle and K.~Naff.
    \newblock Rotational symmetry of ancient solutions to the ricci flow in higher dimensions.
    \newblock {\em Geom. Topol.}, 27(1):153--226, 2023.

    \bibitem[Bre20]{Brendle}
    S.~Brendle.
    \newblock Ancient solutions to the ricci flow in dimension 3.
    \newblock {\em Acta Math.}, 225(1):1--102, 2020.

    \bibitem[BW15]{Bizon-Wasserman}
    P.~Bizo\'{n} and A.~Wasserman.
    \newblock Nonexistence of shrinkers for the harmonic map flow in higher dimensions.
    \newblock {\em Int. Math. Res. Not. IMRN}, (17):7757--7762, 2015.

    \bibitem[BW17]{Bernstein}
    J.~Bernstein and L.~Wang.
    \newblock A topological property of asymptotically conical self-shrinkers of small entropy.
    \newblock {\em Duke Math. J.}, 166(3):403--435, 2017.

    \bibitem[CCMS24]{CHODOSH}
    O.~Chodosh, K.~Choi, C.~Mantoulidis, and F.~Schulze.
    \newblock Mean curvature flow with generic initial data.
    \newblock {\em Invent. Math.}, pages 1--100, 2024.

    \bibitem[CCR15]{Carlotto}
    A.~Carlotto, O.~Chodosh, and Y.~Rubinstein.
    \newblock Slowly converging yamabe flows.
    \newblock {\em Geom. Topol.}, 19(3):1523--1568, 2015.

    \bibitem[CCS23]{CHODOSH-MCF-2}
    O.~Chodosh, K.~Choi, and F.~Schulze.
    \newblock Mean curvature flow with generic initial data {II}.
    \newblock {\em arXiv preprint arXiv:2302.08409}, 2023.

    \bibitem[CDD{\etalchar{+}}22]{choi2022classification}
    B.~Choi, P.~Daskalopoulos, W.~Du, R.~Haslhofer, and N.~Sesum.
    \newblock Classification of bubble-sheet ovals in r4.
    \newblock {\em arXiv preprint arXiv:2209.04931}, 2022.

    \bibitem[CH24]{ChoiB}
    B.~Choi and P.-K. Hung.
    \newblock Thom's gradient conjecture for nonlinear evolution equations.
    \newblock {\em arXiv preprint arXiv:2405.17510}, 2024.

    \bibitem[CHH22]{CHH-Ancient-low-entropy}
    K.~Choi, R.~Haslhofer, and O.~Hershkovits.
    \newblock Ancient low-entropy flows, mean-convex neighborhoods, and uniqueness.
    \newblock {\em Acta Math.}, 228(2):217--301, 2022.

    \bibitem[CHH24]{choi2021nonexistence}
    K.~Choi, R.~Haslhofer, and O.~Hershkovits.
    \newblock A nonexistence result for wing-like mean curvature flows in $\mathbb{R}^4$.
    \newblock {\em Geom. Topol.}, 28(7):3095--3134, 2024.

    \bibitem[CHHW22]{CHHW-ancient-asymptotic-cylinder}
    K.~Choi, R.~Haslhofer, O.~Hershkovits, and B.~White.
    \newblock Ancient asymptotically cylindrical flows and applications.
    \newblock {\em Invent. Math.}, 229(1):139--241, 2022.

    \bibitem[CM12]{CM-genericMCF}
    T.~H. Colding and W.~P. Minicozzi.
    \newblock Generic mean curvature flow {I}; generic singularities.
    \newblock {\em Ann. of Math.}, pages 755--833, 2012.

    \bibitem[CM22]{Choi}
    K.~Choi and C.~Mantoulidis.
    \newblock Ancient gradient flows of elliptic functionals and morse index.
    \newblock {\em Amer. J. Math.}, 144(2):541--573, 2022.

    \bibitem[CS21]{CHODOSH2}
    O.~Chodosh and F.~Schulze.
    \newblock Uniqueness of asymptotically conical tangent flows.
    \newblock {\em Duke Math. J.}, 170(16):3601--3657, 2021.

    \bibitem[CZ15]{Chen-Zhang-entropy}
    Z.~Chen and Y.~Zhang.
    \newblock Stabilities of homothetically shrinking yang-mills solitons.
    \newblock {\em Trans. Amer. Math. Soc.}, 367(7):5015--5041, 2015.

    \bibitem[DeT83]{de-Turck}
    D.~M. DeTurck.
    \newblock Deforming metrics in the direction of their ricci tensors.
    \newblock {\em J. Differential Geom.}, 18(1):157--162, 1983.

    \bibitem[DH24]{du2024hearing}
    W.~Du and R.~Haslhofer.
    \newblock Hearing the shape of ancient noncollapsed flows in $\mathbb{R}^4$.
    \newblock {\em Comm. Pure Appl. Math.}, 77(1):543--582, 2024.

    \bibitem[DHS10]{DHS-classification-compact-ancient-curve-shortening-flow}
    P.~Daskalopoulos, R.~Hamilton, and N.~Sesum.
    \newblock Classification of compact ancient solutions to the curve shortening flow.
    \newblock {\em J. Differential Geom.}, 84(3):455--464, 2010.

    \bibitem[DHS12]{DHS-classification-compact-ancient-Ricci-flow}
    P.~Daskalopoulos, R.~Hamilton, and N.~Sesum.
    \newblock Classification of ancient compact solutions to the ricci flow on surfaces.
    \newblock {\em J. Differential Geom.}, 91(2):171--214, 2012.

    \bibitem[DK97]{Donaldson}
    S.~K. Donaldson and P.~B. Kronheimer.
    \newblock {\em The geometry of four-manifolds}.
    \newblock Oxford Univ. Press, 1997.

    \bibitem[DS19]{DONNINGER}
    R.~Donninger and B.~Sch{\"o}rkhuber.
    \newblock Stable blowup for the supercritical yang-mills heat flow.
    \newblock {\em J. Differential Geom.}, 113(1):55--94, 2019.

    \bibitem[Eck00]{Ecker}
    Klaus Ecker.
    \newblock Logarithmic sobolev inequalities on submanifolds of euclidean space.
    \newblock {\em J. Reine Angew. Math.}, pages 105--118, 2000.

    \bibitem[Eva22]{Evans}
    L.~C. Evans.
    \newblock {\em Partial differential equations}, volume~19.
    \newblock Amer. Math. Soc., 2022.

    \bibitem[Fee14]{Feehan}
    P.~Feehan.
    \newblock Global existence and convergence of solutions to gradient systems and applications to yang-mills gradient flow, 2014.

    \bibitem[Gas02]{Gastel}
    A.~Gastel.
    \newblock Singularities of first kind in the harmonic map and yang-mills heat flows.
    \newblock {\em Math. Z.}, 242:47--62, 2002.

    \bibitem[Gro01]{Grotowski}
    J.~F. Grotowski.
    \newblock Finite time blow-up for the yang-mills heat flow in higher dimensions.
    \newblock {\em Math. Z.}, 237(2):321--333, 2001.

    \bibitem[GS20]{GLOGIC}
    I.~Glogi{\'c} and B.~Sch{\"o}rkhuber.
    \newblock Nonlinear stability of homothetically shrinking yang-mills solitons in the equivariant case.
    \newblock {\em Comm. Partial Differential Equations}, 45(8):887--912, 2020.

    \bibitem[HT04a]{Hong-Tian-Asymptotical-behavior}
    M.-C. Hong and G.~Tian.
    \newblock Asymptotical behavior of the yang-mills flow and singular yang-mills connections.
    \newblock {\em Math. Ann.}, 330:441--472, 2004.

    \bibitem[HT04b]{Hong-Tian}
    M.-C. Hong and G.~Tian.
    \newblock Global existence of the m-equivariant yang-mills flow in four-dimensional spaces.
    \newblock {\em Comm. Anal. Geom.}, 12(1):183--211, 2004.

    \bibitem[Kne80]{Knerr}
    B.~F. Knerr.
    \newblock Parabolic interior schauder estimates by the maximum principle.
    \newblock {\em Arch. Ration. Mech. Anal.}, 75:51--58, 1980.

    \bibitem[KS16]{Kelleher-Street-entropy}
    C.~Kelleher and J.~Streets.
    \newblock Entropy, stability, and yang-mills flow.
    \newblock {\em Commun. Contemp. Math.}, 18(2):1550032, 2016.

    \bibitem[KS18]{Kelleher-generel-blowup}
    C.~Kelleher and J.~Streets.
    \newblock Singularity formation of the yang-mills flow.
    \newblock In {\em Ann. Inst. Henri Poincaré C Anal. Non Linéaire}, volume~35, pages 1655--1686, 2018.

    \bibitem[Nai94]{Naito}
    H.~Naito.
    \newblock Finite time blowing-up for the yang-mills gradient flow in higher dimensions.
    \newblock {\em Hokkaido Math. J.}, 23(3):451--464, 1994.

    \bibitem[Rad92]{Rade}
    J.~Rade.
    \newblock On the yang-mills heat equation in two and three dimensions.
    \newblock {\em J. Reine Angew. Math.}, 426:123--164, 1992.

    \bibitem[RS78]{METHODS-OF-MODERN-MATHEMATICAL-PHYSICS}
    M.~Reed and B.~Simon.
    \newblock {\em Methods of modern mathematical physics; Analysis of operators}.
    \newblock Marcourt Brace Jovanovich, 1978.

    \bibitem[Sch96a]{Schlag}
    W.~Schlag.
    \newblock Schauder and \(l^p\) estimates for parabolic systems via campanato spaces.
    \newblock {\em Comm. Partial Differential Equations}, 21(7--8):1141--1175, 1996.

    \bibitem[Sch96b]{Schlatter-long-time-existence-small-data}
    A.~Schlatter.
    \newblock Global existence of the yang-mills flow in four dimensions.
    \newblock {\em J. Reine Angew. Math.}, 475:133--148, 1996.

    \bibitem[Sch97]{Schlatter-long-time-behavior-4dim}
    A.~Schlatter.
    \newblock Long-time behavior of the yang-mills flow in four dimensions.
    \newblock {\em Ann. Global Anal. Geom.}, 15(1):1--25, 1997.

    \bibitem[Ses04]{Sesum-limiting-behavior}
    N.~Sesum.
    \newblock Limiting behavior of the ricci flow.
    \newblock {\em arXiv preprint}, math/0402194, 2004.

    \bibitem[Ses06a]{Sesum-convergence-ricci-soliton}
    N.~Sesum.
    \newblock Convergence of the ricci flow toward a soliton.
    \newblock {\em Comm. Anal. Geom.}, 14(2):283--343, 2006.

    \bibitem[Ses06b]{Sesum-Ricci}
    N.~Sesum.
    \newblock Linear and dynamical stability of ricci-flat metrics.
    \newblock {\em Duke Math. J.}, 131(1):1--26, 2006.

    \bibitem[Ses08]{Sesum-MCF}
    N.~Sesum.
    \newblock Rate of convergence of the mean curvature flow.
    \newblock {\em Comm. Pure Appl. Math.}, 61(4):464--485, 2008.

    \bibitem[Sim97]{Simon}
    L.~Simon.
    \newblock Schauder estimates by scaling.
    \newblock {\em Calc. Var. Partial Differential Equations}, 5:391--407, 1997.

    \bibitem[SSTZ98]{Schlatter}
    A.~E. Schlatter, M.~Struwe, and A.~S. Tahvildar-Zadeh.
    \newblock Global existence of the equivariant yang-mills heat flow in four space dimensions.
    \newblock {\em Amer. J. Math.}, 120(1):117--128, 1998.

    \bibitem[Str94]{Struwe}
    M.~Struwe.
    \newblock The yang-mills flow in four dimensions.
    \newblock {\em Calc. Var. Partial Differential Equations}, 2:123--150, 1994.

    \bibitem[SWZ22]{Sire-Wei-Zheng}
    Y.~Sire, J.~Wei, , and Y.~Zheng.
    \newblock Infinite time bubbling for the $su(2)$ yang--mills heat flow on $\mathbb{R}^4$.
    \newblock {\em arXiv:2208.13875}, 2022.

    \bibitem[Uhl82]{Uhlenbeck-Kato}
    K.~K. Uhlenbeck.
    \newblock Removable singularities in yang-mills fields.
    \newblock {\em Comm. Math. Phys.}, 83:11--29, 1982.

    \bibitem[Wal14]{Waldron-1}
    A.~Waldron.
    \newblock Self-duality and singularities in the yang-mills flow.
    \newblock {\em Columbia University Thesis}, 2014.

    \bibitem[Wal16]{Waldron-2}
    A.~Waldron.
    \newblock Instantons and singularities in the yang-mills flow.
    \newblock {\em Calc. Var. Partial Differential Equations}, 55:1--31, 2016.

    \bibitem[Wal19]{Waldron}
    A.~Waldron.
    \newblock Long-time existence for yang-mills flow.
    \newblock {\em Invent. Math.}, 217(3):1069--1147, 2019.

    \bibitem[Wei04]{Weinkove}
    B.~Weinkove.
    \newblock Singularity formation in the yang-mills flow.
    \newblock {\em Calc. Var. Partial Differential Equations}, 19(2):211--220, 2004.
\end{thebibliography}
\end{document}